\documentclass[11pt]{article}
\usepackage[letterpaper]{geometry}
\usepackage{amsthm,amsmath,amsfonts}
\usepackage{color,graphicx}
\usepackage{amssymb}
\usepackage{mathrsfs}
\usepackage{subfigure}
\RequirePackage[numbers,sort]{natbib}
\RequirePackage[colorlinks,citecolor=blue,urlcolor=blue]{hyperref}
\newcommand{\R}{\mathbb{R}}
\newcommand{\D}{\mathbb{D}}

\newcommand{\E}{\mathbb{E}}

\newcommand{\DC}{\text{DC}}

\newcommand{\Prob}{\mathbb{P}}

\providecommand{\abs}[1]{\lvert#1\rvert}

\newtheorem{theorem}{Theorem}[section]
\newtheorem{lemma}{Lemma}[section]
\newtheorem{proposition}{Proposition}

\theoremstyle{definition}
\numberwithin{equation}{section}

\newtheorem{assumption}{Assumption}

\begin{document}
\noindent

\begin{titlepage}

\vspace{-1.0in}

\title{\bf Optimal Control of Brownian Inventory Models with Convex
  Inventory Cost:\\ Discounted Cost Case\footnote{Research supported
in part by NSF grants CMMI-0727400, CMMI-0825840, and CMMI-1030589}}

\author{J. G. Dai\footnote
{H. Milton Stewart School of Industrial and Systems Engineering,
  Georgia
  Institute of Technology, Atlanta, Georgia 30332, U.S.A.; dai@gatech.edu}
\ and \ Dacheng Yao\footnote
{Academy of Mathematics and Systems Science, Chinese
  Academy of Sciences, Beijing, 100190, China; dachengyao@amss.ac.cn} }
\date{October 29, 2011}
\maketitle
\thispagestyle{empty}
\begin{abstract}
  We consider an inventory system in which inventory level fluctuates
  as a Brownian motion in the absence of control.  The inventory
  continuously accumulates cost at a rate that is a general convex
  function of the inventory level, which can be negative when there is
  a backlog.  At any time, the inventory level can be adjusted by a
  positive or negative amount, which incurs a fixed positive cost and
  a proportional cost. The challenge is to find an adjustment policy
  that balances the inventory cost and adjustment cost to minimize the
  expected total discounted cost.  We provide a tutorial on using a
  three-step lower-bound approach to solving the optimal control
  problem under a discounted cost criterion. In addition, we prove that a
  four-parameter control band
  policy is optimal among all feasible policies. A key step   is the
  constructive proof of the existence of a unique solution to
  the free boundary problem.  The proof leads naturally to an algorithm
  to compute the four parameters of the optimal control band policy.

\end{abstract}


\medskip

\end{titlepage}

\section{Introduction}
\label{sec:introduction}
Dai and Yao~\cite{DaiYao11a} studied the optimal control of Brownian
inventory models under the \emph{long-run average cost} criterion.
This paper is a companion of \cite{DaiYao11a}. It studies the same
Brownian inventory models, but under the \emph{discounted cost}
criterion. Its main purpose is to provide a tutorial on the powerful,
lower-bound approach to proving the optimality of a control band
policy among all feasible policies. The tutorial is rigorous and,
except the standard It\^o formula, self contained.  In addition, this
paper contributes to the literature by proving the existence of a
``smooth'' solution to the free boundary problem with a general convex
holding cost function. As a consequence, a four-parameter optimal
control band policy is shown to be optimal.  Our existence proof also
leads naturally to an algorithm to compute the optimal control band
parameters.

The introduction in \cite{DaiYao11a} gives  detailed descriptions of
the Brownian inventory models, control band policies, and the lower-bound
approach. It also gives  an extensive literature review. Most of the
development
there including the motivation to study non-linear holding cost
function and  the literature review applies to this paper as well
and it will not be repeated here. In the rest of this introduction, we
highlight the development that is specific to the discounted cost
case.

As in \cite{DaiYao11a}, inventory position is assumed to be adjustable,
either upward or downward. All adjustments are realized immediately
without any leadtime delay.  Each upward adjustment with amount
$\xi>0$ incurs a cost $K+k\xi$, where $K\ge 0$ and $k>0$ are the fixed
cost and the variable cost, respectively, for each upward
adjustment. Similarly, each downward adjustment with amount $\xi$
incurs a cost of $L+\ell \xi$ with fixed cost $L\ge 0$ and variable
cost $\ell>0$.  In addition, we assume that the holding cost function
$h:\R\to \R_+$ is a \emph{general} convex function that satisfies some
minimal assumptions in Assumption \ref{assumption:h-discounted}.  The
objective is to find some control policy that balances the inventory
cost and the adjustment cost so that, starting from any initial
inventory level $x$, the (infinite-horizon) expected total discounted
cost is minimized. When both upward and downward fixed costs are
positive, the model is an \emph{impulse control} problem.  When both fixed
cost are zero, the corresponding Brownian control problem is a
\emph{singular control} or \emph{instantaneous control} problem.  It
was demonstrated in Section 6 of \cite{DaiYao11a} that a singular
control problem is much easier to solve than an impulse control problem
and a two-parameter control band policy is optimal.  This control band
policy can be considered as the limit of a sequence of four-parameter
control band policies, each of which is optimal for an impulse control
problem. Therefore, in this paper, we restrict ourselves to impulse
control problems; namely, we assume that $K>0$ and $L>0$.
Although in this paper  we do not consider the singular control problem or
the Brownian control problem when the inventory backlog is not
allowed,  our proof for the existence of an optimal control band
policy for the impulse control problem can be extended to cover these
two cases. These extensions were carried out in Sections 6 and 7 of
\cite{DaiYao11a} in the average cost setting.

When the inventory holding cost function is linear, namely,
\begin{equation}
  \label{eq:linearh}
  h(x) =
  \begin{cases}
    h_1 x & \text{ if } x \ge 0, \\
    -p_1 x & \text{ if } x <0
  \end{cases}
\end{equation}
for some constants $p_1>0$ and $h_1>0$, Constantinides and Richard
\cite{ConstantinidesRichard78}  proved that a four-parameter control
band policy is optimal under the condition that
\begin{equation}
  \label{eq:costcondition}
  h_1 - \beta k>0 \quad \text{ and } \quad p_1 -\beta \ell>0.
\end{equation}
As explained in \cite{ConstantinidesRichard78} $h_1/\beta$ is the
present value of the holding cost of keeping one unit of inventory now
to infinity. If $h_1/\beta\le \ell$, it will never be optimal to
reduce the inventory level as long as $L>0$. Similarly, if
$p_1/\beta\le k$, it will never be optimal to increase the inventory
as long as $K>0$. Thus, condition (\ref{eq:costcondition}) is also
necessary for a four-parameter control band policy to be
optimal. Baccarin~\cite{Baccarin02} sketched a proof that a
four-parameter control band policy is also optimal when the holding
cost is quadratic given by
\begin{equation}
  \label{eq:quadratic}
  h(x) =
  \begin{cases}
    h_1x + h_2 x^2 & \text{ if } x \ge 0, \\
    -p_1 x + p_2 x^2 & \text{ if } x<0,
  \end{cases}
\end{equation}
where $h_1\ge 0$, $p_1\ge 0$, $h_2>0$ and $p_2>0$. In his proof,
condition (\ref{eq:costcondition}) is not needed any more as long as
$h_2>0$ and $p_2>0$.
Baccarin~\cite{Baccarin02} deferred the detailed proof for the
existence of a solution to the four-parameter free boundary problem
to an online supplement. Unfortunately, this document can no longer be
located over the Internet.
 Assuming $K=L>0$  and $k=\ell=0$,
Plehn-Dujowich~\cite{PlehnDujowich05} proved that a three-parameter control band
policy is optimal when the holding cost function $h$ satisfies
\begin{eqnarray}
  \label{eq:hc1}
&&\text{$h$ and $h'$ are continuous;} \\
&&\text{$h$ is strictly concave and single-peaked;}\\
&&\text{$|h|$, $|h'|$ and $|h''|$ are bounded by a polynomial.}
\end{eqnarray}
Both the linear cost in (\ref{eq:linearh}) and the quadratic cost in
(\ref{eq:quadratic}) do not satisfy the smoothness condition in
(\ref{eq:hc1}).

In this paper, when the holding cost function is assumed to be
general, satisfying Assumption \ref{assumption:h-discounted} in
Section~\ref{sec:model}, we prove that a four-parameter control band
policy is optimal.  Assumption~\ref{assumption:h-discounted} on the
convex holding cost function $h$ is considerable weaker than those in
literature.  The cost functions in
\cite{ConstantinidesRichard78,Baccarin02,PlehnDujowich05} all satisfy
Assumption \ref{assumption:h-discounted}.
Condition (\ref{eq:hlimit}) in Assumption \ref{assumption:h-discounted}
is analogs to (\ref{eq:costcondition}) and is automatically satisfied
for $h$ in (\ref{eq:quadratic}).  Similar to the companion
paper \cite{DaiYao11a}, we adopt the three-step lower-bound approach
in our proof. In the first step, we prove that if there exists a
``smooth'' test function $f$ that satisfies a set of differential
inequalities, the function $f$ dominates the value function at every
initial inventory level $x$.  In the second step, given a control band
policy, it is shown that the value function within the band can be
obtained as the unique solution to a second order differential
equation. In the third step, a solution to a free boundary problem is
shown to exist and satisfy the conditions for $f$ in the first step.

The result in step 1 is known as the ``verification theorem'' in
literature. All three prior papers
\cite{ConstantinidesRichard78,Baccarin02,PlehnDujowich05} invoked the
verification theorem in Richard~\cite{Richard77}, which in turn
generalized the pioneering work of Bensoussan and Lions
\cite{BensoussanLions84,BensoussanLions75}. This tutorial advocates
the lower bound approach that was also adopted by Harrison et.\ al
\cite{HarrisonSellkeTaksar83} and Harrison and Taksar
\cite{HarrisonTaksar83}. The advantage of this approach is that,
except for applying the standard It\^o formula, it is self-contained,
and therefore this approach can readily be rigorously adopted in other
related settings.

The free boundary problem is specified using the well known
``smooth-pasting'' method (see, e.g.,
\cite{BertolaCabellero90}). Solving the free boundary problem in Step
3 is the most difficult task.  We prove the existence of a $C^1$
solution to the free boundary problem that has four free parameters.
Though our proof is similar to the one in
\cite{ConstantinidesRichard78}, where a linear holding cost function
is used, our proof is considerably more difficult. Unlike the proof in
\cite{ConstantinidesRichard78}, our proof is also constructive so that
it leads naturally an algorithm to compute the four parameters of the
optimal control band.  Recently, Feng and Muthuraman
\cite{FengMuthuraman10} developed an algorithm to compute the
parameters of an optimal control band policy for the discounted
Brownian control problem. They illustrated the convergence of their
algorithm through some numerical examples. However, the convergence of
their algorithm was not established.

The rest of this paper is organized as follows.  In Section
\ref{sec:model}, we define our Brownian control problem.
In Section \ref{sec:ito-formula-lower-discounted} we
present a version of It\^o formula
that does not require the test function $f$ be $C^2$ function.  A
lower bound for all feasible policies is established in Section
\ref{sec:lowerbound-discounted}.  Section
\ref{sec:controlBand-discounted}
 shows that under a control band
policy, the value function within the band can be obtained
as a solution to a second order ordinary differential equation (ODE).
Under the assumption that a free-boundary problem has a
unique solution that has desired regularity properties, Section
\ref{sec:optimal-discounted-inpulse}
 proves that there is a
control band policy whose discounted cost achieves
the lower bound. Thus, the
control band policy is optimal among all feasible policies.
Section~\ref{sec:optimal-control-band-discounted} is a lengthy one
that devotes to the construction of the solution to the free-boundary
problem. In the section, the parameters for the optimal control band
policy are characterized.
Section~\ref{sec:optimal-control-band-discounted} constitutes the main
technical contribution of this paper.
%   Finally, Section
% \ref{sec:conclusions} summarizes  this paper and discusses a few extensions.

\section{Impulse Brownian Control Models}
\label{sec:model}

Let $X=\{X(t), t\ge 0\}$ be a Brownian motion
with drift $\mu$ and variance $\sigma^2$, starting from
$x$. Then, $X$ has the following representation
\begin{displaymath}
  X(t) =  x + \mu t + \sigma W(t), \quad t\ge 0,
\end{displaymath}
where $W=\{W(t), t\ge 0\}$ is a standard Brownian
motion that has drift
$0$, variance $1$, starting from $0$.
We assume $W$ is  defined on  some filtered probability space
$(\Omega, \{{\cal F}_t\},  {\cal F}, \Prob)$ and $W$ is an
 $\{{\cal F}_t\}$-martingale. Thus, $W$ is also known as  an  $\{{\cal
   F}_t\}$-standard Brownian motion.
 We use $X$ to model the
\emph{netput process} of the firm.  For each $t\ge 0$, $X(t)$
represents the inventory level at time $t$ if no control has been
exercised by time $t$. The netput process will be controlled and the
actual inventory level at time $t$, after controls has been exercised,
is denoted by $Z(t)$.  The controlled process is denoted by $Z=\{Z(t),
t\ge 0\}$. With a slight abuse of terminology,
we call $Z(t)$ the inventory level at time $t$, although when
$Z(t)<0$, $\abs{Z(t)}$ is the backorder level at time $t$.

Controls are dictated by a policy.  A policy $\varphi$ is a pair of
stochastic processes $(Y_1, Y_2)$ that satisfies the following three
properties: (a) for each sample path $\omega\in \Omega$, $Y_i(\omega,
\cdot)\in \D$, where $\D$ is the set of  functions on
$\R_+=[0,\infty)$ that are right continuous on $[0, \infty)$
and have left limits in $(0, \infty)$, (b) for each $\omega$,
$Y_i(\omega, \cdot)$ is a nondecreasing function, (c) $Y_i$ is adapted
to the filtration $\{{\cal F}_t\}$, namely, $Y_i(t)$ is ${\cal
  F}_t$-measurable for each $t\ge 0$.  We call $Y_1(t)$ and $Y_2(t)$
the cumulative \emph{upward} and \emph{downward} adjustment,
respectively, of the
inventory in $[0, t]$. Under a given policy $(Y_1, Y_2)$, the
inventory level at time $t$ is given by
\begin{equation}
  \label{eq:semimartingaleRep}
  Z(t) = X(t) +Y_1(t) -Y_2(t)=x+\sigma W(t)+\mu t +Y_1(t)-Y_2(t) ,
  \quad t\ge 0.
\end{equation}
Therefore,  $Z$ is a semimartingale, namely,
a martingale $\sigma W$ plus a process that is of bounded variation.

Because $K$ is assumed to be positive,  we restrict upward controls
that have a
finitely many upward adjustment in a finite interval.
This is equivalent to requiring $Y_1$ to be piecewise constant
function on each sample path.  Under such an
upward control, the upward adjustment times can be listed as a
discrete sequence $\{T_1(n):n\ge 0\}$, where the $n$th upward
adjustment time can be defined recursively via
\begin{displaymath}
  T_1(n) = \inf\{t> T_1(n-1): \Delta Y_1(t)>0\},
\end{displaymath}
where, by convention, $T_1(0)=0$ and $\Delta Y_1(t)=Y_1(t)-Y_1(t-)$.
The amount of the $n$th upward adjustment is denoted by
\begin{displaymath}
  \xi_1(n)= Y_1(T_1(n))-Y_1(T_1(n)-) \quad n=0, 1, \ldots.
\end{displaymath}
It is clear that specifying such a upward adjustment policy $Y_1=\{Y_1(t), t\ge
0\}$ is equivalent to specifying a sequence of $\{(T_1(n), \xi_1(n)):
n\ge 0\}$. In particular, given the sequence, one has
\begin{equation}
  \label{eq:Y1N1xi1}
  Y_1(t) = \sum_{i=0}^{N_1(t)} \xi_1(i),
\end{equation}
and $N_1(t)=\max\{n\ge 0:T_1(n)\le t\}$ is the number of upward controls
$[0, t]$. Thus, it is sufficient to specify the sequence
$\{(T_1(n), \xi_1(n)):
n\ge 0\}$  to describe an upward adjustment policy. Similarly, since
$L>0$, it is sufficient to specify the sequence
$\{(T_2(n), \xi_2(n)):
n\ge 0\}$  to describe a downward adjustment policy
and
\begin{equation}
  \label{eq:Y2N2xi2}
  Y_2(t) = \sum_{i=0}^{N_2(t)} \xi_2(i).
\end{equation}
 Merging these two
sequences, we have  the
sequence $\{(T_n, \xi_n), n\ge 0\}$, where
$T_n$ is the $n$th adjustment time of the inventory and $\xi_n$ is the
amount of adjustment at time $T_n$. When $\xi_n>0$, the $n$th
adjustment is an upward adjustment and when $\xi_n<0$, the $n$th
adjustment is a downward adjustment. The policy  $(Y_1, Y_2)$ is
adapted if  $T_n$ is an $\{{\cal F}_{t}\}$-stopping time and
 each adjustment $\xi_n$ must be $\mathscr{F}_{T_n-}$ measurable,
In general, we allow an upward or downward adjustment at time
$t=0$.  By convention,  we set $Z(0-)=x$ and call $Z(0-)$ the
\emph{initial  inventory level}. By (\ref{eq:semimartingaleRep}),
\begin{displaymath}
  Z(0) = x + Y_1(0)-Y_2(0),
\end{displaymath}
which can be different from the initial inventory level $Z(0-)$.

Under a feasible policy $\varphi=\{(Y_1(t), Y_2(t)\}$
with initial inventory level $Z(0-)=x$ and a discount rate $\beta>0$,
the  expected total discounted cost $\DC(x, \varphi)$ is defined to be
\begin{eqnarray}
\lefteqn{\DC(x,\varphi)
=\mathbb{E}_x\Big[\int_0^{\infty} e^{-\beta t} h(Z(t))dt
} \label{eq:DC-1} \\
&& {}
+
\int_0^{\infty} e^{-\beta t}\bigl(K dN_1(t)+L dN_2(t)+k dY_1(t)+\ell dY_2(t)\bigr)
\Big]. \nonumber
\end{eqnarray}
where $\E_x$ is the expectation operator conditioning on the initial
inventory level being $Z(0-)=x$.
Because of (\ref{eq:Y1N1xi1}) and (\ref{eq:Y2N2xi2}), this Brownian
inventory control model is called the impulse Brownian control model.
Clearly,  we  need to restrict our feasible policies to satisfy
\begin{eqnarray}
  \label{eq:regularPolicy-discounted}
  \mathbb{E}_x\Bigl[\sum_{n=0}^{\infty} e^{-\beta
    T_n}\bigl(1+\abs{\xi_n}\bigr)\Bigr]<\infty.
\end{eqnarray}
Otherwise, $\DC(x, \varphi)=\infty$. We assume  the inventory cost function
$h:\R\to \R_+$ satisfies the following assumption.
\begin{assumption}
\label{assumption:h-discounted}
Assume that the cost function $h$ satisfies the following conditions:
(a) it is continuous and convex; (b) there
exists an $a$ such
that $h\in C^2(\R)$ except at $a$ and $h(a)=0$; (c)  $h'(x)\le 0$ for $x<a$ and
$h'(x)\ge 0$ for $x>a$; (d)
\begin{equation}
  \label{eq:hlimit}
  \lim_{ {x}\uparrow\infty} h'(x) >\ell \beta\quad \text{and} \quad
  \lim_{ {x}\downarrow -\infty} h'(x) <-k\beta;
\end{equation}\\
(e)
$h''(x)$ has smaller order than $e^{\lambda_1 x}$ as $x\uparrow \infty$, that is
\begin{eqnarray}
&&\lim_{x\uparrow \infty}\frac{h''(x)}{e^{\lambda_1 x}}=0,\label{eq:h''Lim1-discounted}\\
&&\int_a^{+\infty} e^{-\lambda_1 y}h''(y)dy<\infty,\label{eq:h''Int1-discounted}
\end{eqnarray}
where $\lambda_1= \Bigl[ (\mu^2+ 2\beta\sigma^2)^{1/2}-\mu\Bigr]/\sigma^2>0$.\\
(f) $h''(x)$ has smaller order than $e^{-\lambda_2 x}$ as $x\downarrow -\infty$,
that is
\begin{eqnarray}
&&\lim_{x\downarrow -\infty}\frac{h''(x)}{e^{-\lambda_2 x}}=0,  \label{eq:h''Lim-discounted}\\
&&\int_{-\infty}^a e^{\lambda_2 y}h''(y)dy<\infty. \label{eq:h''Int2-discounted}
\end{eqnarray}
where  $\lambda_2=\Bigl[ (\mu^2+ 2\beta\sigma^2)^{1/2}+\mu\Bigr]/\sigma^2>0$
\end{assumption}

\begin{remark}
(a) If $h$ is given by \eqref{eq:linearh}, \eqref{eq:hlimit}
becomes \eqref{eq:costcondition}, which is consistent with (13) in
\cite{ConstantinidesRichard78}. (b) The continuous and convex
  holding cost
function $h$ can be relaxed to be  continuously differentiable once
and twice at all but a finitely many points.
(c) When $\lim_{x\to\infty}h'(x)\le \ell \beta$, it follows the
  same reasoning as in \cite{ConstantinidesRichard78} that  it will
  never be optimal to reduce the inventory level as long as $L>0$. Similarly,
 when $\lim_{x\to-\infty}h'(x)\ge  k\beta$, it will
  never be optimal to increase the inventory level as long as $K>0$.
\end{remark}

The following elementary lemma on the holding cost function is useful
in later development.
\begin{lemma} \label{lem:hproperty}
(a) Under Assumption~\ref{assumption:h-discounted},
\begin{eqnarray}
\label{eq:limint-discounted}
&&\lim_{x\downarrow -\infty}\frac{\int_x^a e^{-\lambda_1
    (y-a)}h''(y)dy}{e^{-(\lambda_1+\lambda_2) (x-a)}}
=0, \\
\label{eq:limint-discounted2}
&& \lim_{x\uparrow \infty}\frac{\int^x_a e^{\lambda_2
    (y-a)}h''(y)dy}{e^{(\lambda_1+\lambda_2) (x-a)}}
=0.
\end{eqnarray}
(b)  Under Assumption~\ref{assumption:h-discounted},
\begin{eqnarray}
&& \lim_{x\uparrow \infty}\frac{\lambda_2\int_a^{x} e^{\lambda_2
    (y-a)}h'(y)dy}{ e^{\lambda_2 (x-a)}}=\lim_{x\uparrow \infty}
h'(x),\label{eq:lim=h'-discounted2}  \\
&& \lim_{x\downarrow -\infty}\frac{\lambda_1\int^a_{x} e^{-\lambda_1
    (y-a)}h'(y)dy}{ e^{-\lambda_1 (x-a)}}=\lim_{x\uparrow -\infty}
h'(x).\label{eq:lim=h'-discounted3}
\end{eqnarray}

\end{lemma}
\begin{proof}
(a) We prove \eqref{eq:limint-discounted}. The proof of
\eqref{eq:limint-discounted2}  is similar and is omitted.
 If
 \begin{displaymath}
\lim_{x\downarrow -\infty}\int_x^a e^{-\lambda_1 (y-a)}h''(y)dy<\infty,
 \end{displaymath}
\eqref{eq:limint-discounted} clearly holds.
Now assume that
\begin{equation*}
\lim_{x\downarrow -\infty}\int_x^a e^{-\lambda_1 (y-a)}h''(y)dy=\infty.
\end{equation*}
By using the \emph{L' H\^{o}pital rule}, one has
\begin{eqnarray*}
\lim_{x\downarrow -\infty}\frac{\int_x^a e^{-\lambda_1 (y-a)}h''(y)dy}{e^{-(\lambda_1+\lambda_2) (x-a)}}
&=&\lim_{x\downarrow -\infty}\frac{-e^{-\lambda_1 (x-a)}h''(x)}{-(\lambda_1+\lambda_2)e^{-(\lambda_1+\lambda_2) (x-a)}}\\
&=&\lim_{x\downarrow -\infty}\frac{h''(x)}{(\lambda_1+\lambda_2)e^{-\lambda_2 (x-a)}}\\
&=&0,
\end{eqnarray*}
where the last equality is due to \eqref{eq:h''Lim-discounted}.

(b) We prove (\ref{eq:lim=h'-discounted2}). The proof of
(\ref{eq:lim=h'-discounted2})  is similar and is omitted.

The first part of \eqref{eq:hlimit} implies that
there exist a constant $c_1>0$ and $x''\in(a,\infty)$ such that
for any $x\geq x''$,
\begin{eqnarray*}
h'(x)\geq c_1,
\end{eqnarray*}
which yields that
\begin{eqnarray}
\lim_{x\uparrow \infty} \int_a^{x} e^{\lambda_2 (y-a)}h'(y)dy
&\geq& \lim_{x\uparrow \infty} \int_{x''}^{x} e^{\lambda_2 (y-a)}h'(y)dy\nonumber\\
&\geq&c_1\cdot \lim_{x\uparrow \infty} \int_{x''}^{x} e^{\lambda_2 (y-a)}dy\nonumber\\
&=&\infty,\label{eq:lim=infty-discounted}
\end{eqnarray}
where the first inequality is due to the assumption $h'(x)\geq 0$ for
$x>a$.
By using the \emph{L' H\^{o}pital rule}, one has
\begin{eqnarray*}
 \lim_{x\uparrow \infty}\frac{\lambda_2\int_a^{x} e^{\lambda_2
    (y-a)}h'(y)dy}{ e^{\lambda_2 (x-a)}}
&=&\lim_{x\uparrow \infty}\frac{\lambda_2 e^{\lambda_2
      (x-a)}h'(x)}{ \lambda_2e^{\lambda_2 (x-a)}}\nonumber\\
&=&\lim_{x\uparrow \infty}
h'(x).
\end{eqnarray*}
\end{proof}

\section{It\^{o} Formula}
\label{sec:ito-formula-lower-discounted}
In this section, we present the It\^{o} formula, tailored to
the discounted setting.
\begin{lemma}
\label{lem:Ito-discounted}
Assume that $f\in C^1(\R)$ and $f'$ is absolutely continuous such that
$f'(b)-f'(a)=\int_a^bf''(u)du$ for any $a<b$ with
$f''$ locally in $L^1(\R)$.
Then
\begin{eqnarray}
e^{-\beta t}f(Z(t)) & =&
f(Z(0))+\int_0^t e^{-\beta s}\bigl(\Gamma f(Z(s))-\beta f(Z(s))\bigr)ds
 \label{eq:ito-discouned}\\
&&{}+ \sigma \int_0^t e^{-\beta s}  f'(Z(s))dW(s)+ \sum_{0< s \le t}e^{-\beta s} \Delta
f(Z(s)), \nonumber
\end{eqnarray}
where
\begin{equation}
  \label{eq:Gamma}
\Gamma f(x) = \frac{1}{2}\sigma^2 f''(x) + \mu f'(x)  \quad \text{ for
each } x\in \R \text{ such that $f''(x)$ exists},
\end{equation}
is the generator of the $(\mu, \sigma^2)$-Brownian motion $X$, and
$\int_0^t e^{-\beta t} f'(Z(s))dW(s)$ is interpreted as the It\^{o}
integral.
\end{lemma}
\begin{proof}
Using (3.2) in \cite{DaiYao11a} and the integration by parts formula for
semimartingales (see, for example, Page 83 of \cite{Protter05}),
we  have \eqref{eq:ito-discouned}.
\end{proof}

\section{Lower Bound}
\label{sec:lowerbound-discounted}
In this section, we state and prove a theorem that establishes a lower bound for
the  optimal expected total discounted cost. This theorem is closely related to the
``verification theorem" in literature. Its proof is self contained, using the It\^{o} formula
in Section \ref{sec:ito-formula-lower-discounted}.

Define
\begin{eqnarray}
\label{eq:phi-discounted}
\phi(\xi)= \left \{
\begin{array}{ll}
   K+k\xi & \mbox{if }\xi>0, \\
   0 & \mbox{if }\xi=0, \\
   L-l\xi & \mbox{if }\xi<0.
  \end{array}
\right.
\end{eqnarray}

\begin{theorem}\label{thm:lowerbound-discounted}
  Suppose that $f\in C^1(\R)$ and $f'$ is absolutely continuous with $f''$
  locally in $L^1(\R)$. Suppose that there exists a constant $M>0$ such
  that  $|f'(x)|\leq M $ for all $x\in \R$.
 Assume further that
  \begin{eqnarray}
    && \Gamma f(x)-\beta f(x)+h(x)\geq 0 \mbox{ for almost all $x\in
      \mathbb{R}$},\label{eq:lbPoission-discounted}\\
    && f(y) - f(x)\le K+ k(x-y) \mbox{ for $y<x$},\label{eq:lbK-discounted} \\
    && f(y) - f(x)\le L+\ell(y-x) \mbox{ for $x<y$}.\label{eq:lbL-discounted}
  \end{eqnarray}
Then $\DC(x,\varphi)\geq f(x)$ for each feasible policy $\varphi$  and each
initial state $Z(0-)=x\in \mathbb{R}$.
\end{theorem}
\begin{proof}
By It\^{o} formula \eqref{eq:ito-discouned},
\begin{eqnarray}
e^{-\beta t}f(Z(t)) & =&
f(Z(0-))+\int_0^t e^{-\beta s}\bigl(\Gamma f(Z(s))-\beta f(Z(s))\bigr)ds
\nonumber \\
&&{}+ \sigma \int_0^t e^{-\beta s}  f'(Z(s))dW(s)+ \sum_{0\le s \le t}e^{-\beta s} \Delta
f(Z(s))\nonumber\\
&\ge& f(Z(0-))-\int_0^t e^{-\beta s}h(Z(s))ds
+ \sigma \int_0^t e^{-\beta s}  f'(Z(s))dW(s) \nonumber\\
&&{}+ \sum_{0\le s \le t}e^{-\beta s} \Delta
f(Z(s)),\label{eq:itoinequality-discounted}
\end{eqnarray}
where the inequality is due to \eqref{eq:lbPoission-discounted}.
By (\ref{eq:Y1N1xi1}) and (\ref{eq:Y2N2xi2}), the control $\{(Y_1(t),
Y_2(t)), t\ge 0\}$ is equivalent to specifying a sequence $\{(T_n, \xi_n): n=0, 1,
\ldots \}$.
 Conditions (\ref{eq:lbK-discounted}) and (\ref{eq:lbL-discounted}) imply
that
 and $\Delta f (Z(T_n))\ge
- \phi(\xi_n)$, where $\phi$ is given by \eqref{eq:phi-discounted}.
Therefore, (\ref{eq:itoinequality-discounted}) leads to
\begin{eqnarray}
  \label{eq:itoinequalityDiscounted2}
\lefteqn{ e^{-\beta t}f(Z(t))} \\
&& {} \ge f(Z(0-)) -\int_0^t e^{-\beta s}h(Z(s))ds + \sigma \int_0^t
  e^{-\beta s}f'(Z(s)) dW(s) - \sum_{n=0}^{N(t)} e^{-\beta T_n}\phi(\xi_n)\nonumber
\end{eqnarray}
for each $t\ge 0$.
Fix an $x\in \R$. We assume that
\begin{displaymath}
  \E_x\biggl ( \int_0^t e^{-\beta s}h (Z(s))ds + \sum_{n=0}^{N(t)}e^{-\beta T_n}\phi(\xi_n) \biggr )< \infty
\end{displaymath}
for each $t>0$. Otherwise, $\DC(x, \varphi )=\infty$ and $\DC(x,
\varphi)\ge f(x)$  is trivially satisfied. Because $\abs{f'(x)}\le
M$, one has
$\E_x{\int_0^te^{-\beta s}f'(Z(s))dW(s)}=0$. Meanwhile
\begin{displaymath}
f(Z(t)) \le \bigl
(f(Z(t)) \bigr )^+
\end{displaymath}
and $\E_x \bigl[e^{-\beta t}\bigl(f(Z(t)) \bigr )^+\bigl]$ is well defined, though
it can be $\infty$,  where, for a $b\in \R$, $b^+=\max(b, 0)$.
Taking $\E_x$ on the both
 sides of (\ref{eq:itoinequalityDiscounted2}) and noting $f(Z(0-))=f(x)$, we have
 \begin{displaymath}
  \E_x\bigl[e^{-\beta t}\bigl(f(Z(t)) \bigr )^+\bigl]\ge  f(x)-
  \E_x\biggl ( \int_0^t e^{-\beta s}h (Z(s))ds + \sum_{n=0}^{N(t)}e^{-\beta T_n}\phi(\xi_n) \biggr ).
 \end{displaymath}
Taking limit as $t\to\infty$, one has
\begin{equation}
  \label{eq:interInequalityDiscounted}
\liminf_{t\to\infty} \biggl[ \E_x\biggl ( \int_0^te^{-\beta s}h (Z(s))ds +
\sum_{n=0}^{N(t)}e^{-\beta T_n}\phi(\xi_n) \biggr )  +  \E_x\bigl[e^{-\beta t}\bigl(f(Z(t)) \bigr )^+\bigl] \biggr] \ge f(x).
\end{equation}

The boundedness of $f'$ implies that
\begin{eqnarray*}
\bigl(f(x)\bigr)^+\le M(1+\abs{x}),
\end{eqnarray*}
which further implies that
\begin{eqnarray}
\label{eq:f<-discounted}
\bigl(f(Z(t)) \bigr )^+\le M(1+\abs{Z(t)})\le M(1+\abs{x}+\abs{\mu}t+\sigma\abs{W(t)}+\sum_{n=0}^{N(t)}\abs{\xi_n}).
\end{eqnarray}
The following arguments follow the ones  on Page 842 of
\cite{FengMuthuraman10}.
Let $\nu(t)=\sum_{n=0}^{N(t)}\abs{\xi_n}$. Then
\eqref{eq:regularPolicy-discounted} implies
\begin{eqnarray*}
\E_x\Big[\int_0^{\infty} e^{-\beta t} d \nu(t)\Big]<\infty.
\end{eqnarray*}
From (7.5) of Taksar \cite{Taksar97}, we have
\begin{eqnarray*}
\E_x\Big[\int_0^{\infty} e^{-\beta t} \nu(t)dt \Big]  \le
\frac{1}{\beta}\E_x\Big[\int_0^{\infty} e^{-\beta t} d \nu(t)\Big]
<\infty.
\end{eqnarray*}
Applying Fubini's theorem, we have
\begin{eqnarray*}
\int_0^{\infty} e^{-\beta t} \E_x\bigl[\nu(t)\bigr]dt <\infty,
\end{eqnarray*}
which, together with Lemma 4.1 of \cite{FengMuthuraman10}, implies
\begin{eqnarray*}
\liminf_{t\to\infty} e^{-\beta t} \E_x[\nu(t)]=0.
\end{eqnarray*}
Therefore, \eqref{eq:f<-discounted} implies that
\begin{eqnarray*}
\liminf_{t\to\infty} \E_x\bigl[e^{-\beta t}\bigl(f(Z(t)) \bigr )^+\bigl]
\le \liminf_{t\to\infty} \E_x\bigl[e^{-\beta t}\bigl(M(1+\abs{x}+\abs{\mu}t+\sigma\abs{W(t)}+\nu(t) \bigr )\bigl]
=0.
\end{eqnarray*}
\end{proof}

\section{Control Band Policies}
\label{sec:controlBand-discounted}
We use  $\{d, D, U, u\}$ to denote the
control band policy associated with parameters $d$, $D$, $U$, and $u$
with
$d<D< U<u$.
Let us fix a control band policy $\varphi=\{d, D, U, u\}$
and an initial inventory level $Z(0-)=x$.
The adjustment amount $\xi_n$ of
the control band policy is given by
\begin{displaymath}
 \xi_0=
   \begin{cases}
   D-x & \mbox{if } x\leq d, \\
   0 & \mbox{if } d<x<u, \\
   U-x & \mbox{if }x\geq u,
   \end{cases}
\end{displaymath}
and for $n=1,2,...$,
\begin{displaymath}
  \xi_n=
  \begin{cases}
   D-d & \mbox{if }Z(T_n-)=d, \\
   U-u & \mbox{if }Z(T_n-)=u,
  \end{cases}
\end{displaymath}
where again $Z({t-})$ denotes the left limit at time $t$, $T_0=0$ and
\begin{displaymath}
  T_n = \inf\bigl\{ t> T_{n-1}: Z(t)\in\{d, u\}\bigr\}
\end{displaymath}
is the $n$th adjustment time.  (By convention, we assume $Z$ is right
continuous having left limits.)  Our first task is to obtain an
expression for the \emph{value function} $\bar V$, where $\bar V(x)$ is the
expected total discounted cost when the initial inventory level is
$x$. We first present the following lemma.
\begin{theorem}
\label{thm:control band-discounted}
Assume that we fix a control band policy $\varphi=\{d, D, U, u\}$. If
there
exists a twice continuously differentiable
function
$V:[d,u]\rightarrow \R$ that satisfies
\begin{equation}
\Gamma V(x)-\beta V(x)+h(x)=0 \quad d\leq x\leq u,\label{eq:Poisson equ-discounted}
\end{equation}
with boundary conditions
\begin{eqnarray}
 && V(d)-V(D)=K+k(D-d),\label{eq:V(d)-V(D)-discounted}\\
 && V(u)-V(U)=L+l(u-U),\label{eq:V(u)-V(U)-discounted}
\end{eqnarray}
then for each starting point $x\in \R$,
the expected total discounted cost $DC(x,\varphi)$  is given by
\begin{eqnarray}
\label{eq:barV-discounted}
\bar{V}(x)=
\left\{
\begin{array}{ll}
V(D)+K+k(D-x) &\text{for $x\in(-\infty,d]$},\\
V(x) & \text{for $x\in(d,u)$},\\
V(U)+L-\ell(U-x) & \text{for $x\in[u,\infty)$},
\end{array}
\right.
\end{eqnarray}
where $V(x)$ is in \eqref{eq:Poisson equ-discounted}.
\end{theorem}
\begin{remark}
\eqref{eq:V(d)-V(D)-discounted} and \eqref{eq:V(u)-V(U)-discounted} imply that
$\bar{V}$ is continuous at $d$ and $u$.
\end{remark}
\begin{proof}
Consider the control band policy $\varphi=\{d, D, U, u\}$. Let $V$
be a twice continuously differentiable function on $[d, u]$ that
satisfies (\ref{eq:Poisson equ-discounted})-(\ref{eq:V(u)-V(U)-discounted}).
Because $d\le Z(t)\le u$, by Lemma~\ref{lem:Ito-discounted},  we have
\begin{displaymath}
\mathbb{E}_x[e^{-\beta t}V(Z(t))]=\mathbb{E}_x[V(Z(0))]+\mathbb{E}_x\Big[\int_0^t e^{-\beta s}\bigl(\Gamma
V(Z(s))-\beta V(Z(s))\bigr)ds\Big] +\mathbb{E}_x\Big[\sum_{n=1}^{N(t)}e^{-\beta T_n}\theta_n \Big],
\end{displaymath}
where $\theta_n=V(Z(T_n))-V(Z(T_n-))$. Boundary conditions
(\ref{eq:V(d)-V(D)-discounted}) and
(\ref{eq:V(u)-V(U)-discounted}) imply that
$\theta_n=V(Z(T_n))-V(Z(T_n-))=-\phi(\xi_n)$ for $n\geq 1$. Therefore,
\begin{eqnarray*}
&&\mathbb{E}_x[e^{-\beta t}V(Z(t))]-\mathbb{E}_x[V(Z(0))]\nonumber\\
&&\quad\quad=\mathbb{E}_x\Big[\int_0^t e^{-\beta s}\bigl(\Gamma
V(Z(s))-\beta V(Z(s))\bigr)ds\Big] +\mathbb{E}_x\Big[\sum_{n=1}^{N(t)}e^{-\beta T_n}\theta_n
\Big]\nonumber\\
&&\quad\quad=-\mathbb{E}_x\Big[\int_0^t e^{-\beta s} h(Z(s))ds\Big]
-\mathbb{E}_x\Big[\sum_{n=1}^{N(t)}e^{-\beta T_n}\phi(\xi_n) \Big] \nonumber\\
&&\quad\quad=-\mathbb{E}_x\Big[\int_0^t e^{-\beta s} h(Z(s))ds\Big]
-\mathbb{E}_x\Big[\sum_{n=0}^{N(t)}e^{-\beta T_n}\phi(\xi_n) \Big]+\mathbb{E}_x[\phi(\xi_0)].
\end{eqnarray*}
Letting $t\rightarrow\infty$, we have
\begin{eqnarray}
\label{eq:DC=V-discounted}
\DC(x,\varphi)=\mathbb{E}_x[V(Z(0))]+\mathbb{E}_x[\phi(\xi_0)]
\end{eqnarray}
because
\begin{displaymath}
\lim_{t\rightarrow\infty}\mathbb{E}_x[e^{-\beta t}V(Z(t))]=0.
\end{displaymath}
If $Z(0-)=x\in(d,u)$, we have $Z(0)=Z(0-)=x$ and $\xi_0=0$, then
\begin{eqnarray*}
\DC(x,\varphi)=V(x).
\end{eqnarray*}
If $x\leq d$, under control band policy $\varphi=\{d, D, U, u\}$, $Z$ immediately jumps up to $D$.
Therefore, $Z(0)=D$ and $\xi_0=D-x$, then
\begin{eqnarray*}
\mathbb{E}_x[V(Z(0))]=V(D),\quad \mathbb{E}_x[\phi(\xi_0)]=\phi(D-x)=K+k(D-x),
\end{eqnarray*}
which, together with \eqref{eq:DC=V-discounted}, implies that
\begin{eqnarray*}
\DC(x,\varphi)=V(D)+K+k(D-x).
\end{eqnarray*}
The analysis for the case $x\geq u$ is analogous and is omitted.
\end{proof}

We end this section by explicitly finding a solution $V$ to \eqref{eq:Poisson equ-discounted}-\eqref{eq:V(u)-V(U)-discounted}.
\begin{proposition}
\label{prop:V-discounted}
Let $\varphi=\{d,D,U,u\}$ be a control band policy with
\begin{displaymath}
d<D<U<u.
\end{displaymath}
Define
\begin{eqnarray*}
V(x)=  A_1 e^{\lambda_1 x } + B_1 e^{-\lambda_2 x} + V_0(x),
\end{eqnarray*}
where
\begin{eqnarray}
&&V_0(x) = \frac{2}{\sigma^2}
\frac{1}{\lambda_1+\lambda_2}\biggl[
\int_a^x e^{-\lambda_2(x-y)}h(y)dy -
\int_a^x e^{\lambda_1(x-y)}h(y)dy  \biggr],\label{eq:V0-discounted}\\
&&A_1=\frac{b_2\bigl(V_0(D)-V_0(d)+K+k(D-d)\bigr)-b_1\bigl(V_0(U)-V_0(u)+L+\ell(u-U)\bigr)}{a_1b_2-a_2b_1},\mbox{\ \qquad
}\label{eq:A1-discounted}\\
&&B_1=\frac{a_2\bigl(V_0(D)-V_0(d)+K+k(D-d)\bigr)-a_1\bigl(V_0(U)-V_0(u)+L+\ell(u-U)\bigr)}{a_2b_1-a_1b_2}.\text{\
}\label{eq:B1-discounted}
\end{eqnarray}
Then $V$ is a solution to \eqref{eq:Poisson equ-discounted}-\eqref{eq:V(u)-V(U)-discounted}.
In \eqref{eq:A1-discounted} and \eqref{eq:B1-discounted}, we set
\begin{eqnarray}
&&a_1=e^{\lambda_1 d}-e^{\lambda_1 D}, \quad a_2=e^{\lambda_1 u}-e^{\lambda_1 U},\label{eq:cof1-discounted}\\
&&b_1=e^{-\lambda_2 d}-e^{-\lambda_2 D}, \quad b_2=e^{-\lambda_2 u}-e^{-\lambda_2 U}.\label{eq:cof2-discounted}
\end{eqnarray}
\end{proposition}
\begin{proof}
Let
\begin{eqnarray*}
  && \lambda_1 = \Bigl[ (\mu^2+ 2\beta\sigma^2)^{1/2}
  -\mu\Bigr]/\sigma^2>0, \\
  && \lambda_2 = \Bigl[ (\mu^2+ 2\beta\sigma^2)^{1/2}
  +\mu\Bigr]/\sigma^2>0,
\end{eqnarray*}
so that $z=\lambda_1$ and $z=-\lambda_2$ are two solutions of the
quadratic equation
\begin{displaymath}
  \frac{1}{2}\sigma^2 z^2 + \mu z - \beta=0.
\end{displaymath}
The homogenous ordinary differential equation (ODE)
\begin{displaymath}
  \Gamma g - \beta g=0
\end{displaymath}
has two independent solutions $g_1(x)$ and $g_2(x)$, where
\begin{displaymath}
 g_1(x)= e^{\lambda_1 x} \quad \text{and} \quad g_2(x) =e^{-\lambda_2
   x}.
\end{displaymath}
Let
\begin{displaymath}
  w(x) = \det
  \begin{pmatrix}
    g_1(x) & g_2(x) \\
    g_1'(x) & g_2'(x)
  \end{pmatrix}
=-(\lambda_1+\lambda_2) e^{(\lambda_1-\lambda_2)x}\neq 0
\end{displaymath}
and
\begin{eqnarray*}
&&  a_1(x)= \int_a^x \frac{1}{w(y)} g_2(y)
\frac{2}{\sigma^2}h(y)dy=-\frac{1}{\lambda_1+\lambda_2}\frac{2}{\sigma^2}\int_a^x
e^{-\lambda_1y}h(y)dy, \\
&&  a_2(x)= - \int_a^x \frac{1}{w(y)} g_1(y)
\frac{2}{\sigma^2} h(y)dy
=\frac{1}{\lambda_1+\lambda_2}\frac{2}{\sigma^2}\int_a^x e^{\lambda_2y}h(y)dy,
\end{eqnarray*}
where $a$ is the minimum point of the convex inventory cost function $h$.
Then the non-homogenous ODE (\ref{eq:Poisson equ-discounted})  has a
particular solution $V_0(x)$
 given by
\begin{eqnarray*}
V_0(x) = \bigl[a_1(x) g_1(x)+ a_2(x) g_2(x)\bigr] =
\frac{2}{\sigma^2}
\frac{1}{\lambda_1+\lambda_2}\biggl[
\int_a^x e^{-\lambda_2(x-y)}h(y)dy -
\int_a^x e^{\lambda_1(x-y)}h(y)dy  \biggr].
\end{eqnarray*}
A general solution $V(x)$ to (\ref{eq:Poisson equ-discounted}) is given
by
\begin{equation*}
V(x)=  A_1 e^{\lambda_1 x } + B_1 e^{-\lambda_2 x} + V_0(x).
\end{equation*}
Boundary conditions
(\ref{eq:V(d)-V(D)-discounted}) and (\ref{eq:V(u)-V(U)-discounted}) become
\begin{eqnarray}
\bigl(A_1e^{\lambda_1 d}+B_1e^{-\lambda_2 d}+V_0(d)\bigr)-\bigl(A_1e^{\lambda_1 D}+B_1e^{-\lambda_2 D}+V_0(D)\bigr)=K+k(D-d),\label{eq:V(d)-V(D)2-discounted}\\
\bigl(A_1e^{\lambda_1 u}+B_1e^{-\lambda_2 u}+V_0(u)\bigr)-\bigl(A_1e^{\lambda_1 U}+B_1e^{-\lambda_2 U}+V_0(U)\bigr)=L+\ell(u-U).\label{eq:V(u)-V(U)2-discounted}\
\end{eqnarray}
Using the coefficients defined in \eqref{eq:cof1-discounted}-\eqref{eq:cof2-discounted},
we see the boundary conditions \eqref{eq:V(d)-V(D)2-discounted} and \eqref{eq:V(u)-V(U)2-discounted}
become
\begin{eqnarray*}
&&A_1a_1+B_1b_1+V_0(d)-V_0(D)=K+k(D-d),\\
&&A_1a_2+B_1b_2+V_0(u)-V_0(U)=L+\ell(u-U),
\end{eqnarray*}
from which we have unique solution for $A_1$ and $B_1$ given in \eqref{eq:A1-discounted} and \eqref{eq:B1-discounted}.
\end{proof}

\section{Optimal Policy and Optimal Parameters}
\label{sec:optimal-discounted-inpulse}
Theorem \ref{thm:lowerbound-discounted} suggests the following strategy to obtain an
optimal policy. We hope that a control band policy is optimal. Therefore, the first
task is to find an optimal policy among all control band policies. We denote this optimal
control band policy by $\varphi^*=\{d^*,D^*,U^*,u^*\}$ with the expected total discounted
cost
\begin{eqnarray}
\label{eq:barV2-discounted}
\bar{V}(x)=
\left\{
\begin{array}{ll}
V(D^*)+K+k(D^*-x) &\text{for $x\in(-\infty,d^*]$},\\
V(x) & \text{for $x\in(d^*,u^*)$},\\
V(U^*)+L-\ell(U^*-x) & \text{for $x\in[u^*,\infty)$},
\end{array}
\right.
\end{eqnarray}
for any starting point $x\in \R$.
We hope that $\bar{V}$ can be used as the function
$f$ in Theorem \ref{thm:lowerbound-discounted}.
To find the corresponding $f$ that satisfies all the conditions of Theorem \ref{thm:lowerbound-discounted},
we provide the conditions that should be imposed on the optimal parameters
\begin{eqnarray}
\label{eq:optimalconditions1-discounted}
&& V'(D^*)=-k,\quad
V'(U^*)=\ell, \\
\label{eq:optimalconditions2-discounted}
&& V'(d^*)=-k,\quad
V'(u^*)=\ell.
\end{eqnarray}
See Section~5.2 of \cite{DaiYao11a} for an intuitive derivation of
these conditions.  Under condition
(\ref{eq:optimalconditions2-discounted}), $\bar V$ is  a $C^1$ function
on $\R$. Therefore, (\ref{eq:optimalconditions2-discounted}) is also
known as the ``smooth-pasting'' condition.

In this section, we will first prove in Theorem \ref{thm:optimalParameters-discounted}
the existence of parameters $d^*$, $D^*$, $U^*$ and $u^*$ such that the
value function $V$, defined on $[d^*,u^*]$, corresponding the control band policy
$\varphi^*=\{d^*,D^*,U^*,u^*\}$ satisfies \eqref{eq:Poisson equ-discounted}-\eqref{eq:V(u)-V(U)-discounted}
and \eqref{eq:optimalconditions1-discounted}-\eqref{eq:optimalconditions2-discounted}.
As part of the solution, we are to find the boundary points $d^*$, $D^*$, $U^*$ and $u^*$ from
equations \eqref{eq:Poisson equ-discounted}-\eqref{eq:V(u)-V(U)-discounted}
and \eqref{eq:optimalconditions1-discounted}-\eqref{eq:optimalconditions2-discounted}.
These equations define a {\em free boundary problem}.
We then prove in Theorem \ref{thm:optimal-discounted} that the function $\bar{V}$ in \eqref{eq:barV2-discounted}
with parameters $d^*$, $D^*$, $U^*$ and $u^*$
satisfies all the conditions in Theorem \ref{thm:lowerbound-discounted};
therefore, the control band policy $\varphi^*$ is optimal among all feasible policies.

To facilitate the presentation of Theorem \ref{thm:optimalParameters-discounted}, we first find
a general solution without worrying about boundary conditions \eqref{eq:V(d)-V(D)-discounted} and \eqref{eq:V(u)-V(U)-discounted}.
Proposition \ref{prop:V-discounted} shows that $V$ is given in the form
\begin{eqnarray}
\label{eq:V-discounted}
V(x)=  A_1 e^{\lambda_1 x } + B_1 e^{-\lambda_2 x} + V_0(x)
\quad \text{for $x\in\R$,}
\end{eqnarray}
where $V_0$ is given in \eqref{eq:V0-discounted}.
Since $A_1$ and $B_1$ are yet to be determined, which need $d^*$, $D^*$, $U^*$ and $u^*$,
$V$ is also yet to be determined.
Differentiating both sides of \eqref{eq:Poisson equ-discounted} with respect to $x$,
we have that
\begin{eqnarray}
g(x)&=& V'(x)\label{eq:g1-discounted}\\
 &=& \lambda_1 A_1 e^{\lambda_1 x } -\lambda_2 B_1
e^{-\lambda_2 x}  \nonumber\\
&&-\frac{2}{\sigma^2}\frac{1}{\lambda_1+\lambda_2}\biggl[
\lambda_2 \int_a^x e^{-\lambda_2(x-y)}h(y)dy +
\lambda_1 \int_a^x e^{\lambda_1(x-y)}h(y)dy\biggr]\nonumber
\end{eqnarray}
is a solution to
\begin{eqnarray}
\label{eq:gPoisson-discounted}
\Gamma g(x)-\beta g(x)+h'(x)=0 \quad \text{for all $x\in \R\setminus \{a\}$}.
\end{eqnarray}
$g(x)$ in \eqref{eq:g1-discounted} can be rewritten as
\begin{eqnarray}
g(x)&=&  V'(x)\nonumber\\
 &=& \lambda_1 A_1 e^{\lambda_1 x } -\lambda_2 B_1
e^{-\lambda_2 x}\nonumber\\
&&-\frac{2}{\sigma^2}\frac{1}{\lambda_1+\lambda_2}\biggl[
\lambda_2 \int_a^x e^{-\lambda_2(x-y)}h(y)dy +
\lambda_1 \int_a^x e^{\lambda_1(x-y)}h(y)dy\biggr]\nonumber \\
 &=& \lambda_1 A_1 e^{\lambda_1 x } -\lambda_2 B_1
e^{-\lambda_2 x}\nonumber\\
&&-\frac{2}{\sigma^2}\frac{1}{\lambda_1+\lambda_2}\biggl[
 - \int_a^x e^{-\lambda_2(x-y)}h'(y)dy +
  \int_a^x e^{\lambda_1(x-y)}h'(y)dy\biggr] \nonumber\\
&=& \frac{2}{\sigma^2}\frac{1}{\lambda_1+\lambda_2}\biggl[
 \frac{1}{\lambda_1}\bigl(A-\lambda_1\int_a^x e^{-\lambda_1 (y-a)}h'(y)dy\bigr)e^{\lambda_1 (x-a)}
\nonumber \\
&& {} +\frac{1}{\lambda_2}\bigl(B+\lambda_2\int_a^x e^{\lambda_2 (y-a)}h'(y)dy\bigr) e^{-\lambda_2 (x-a)}
  \biggr], \label{eq:g2-discounted}
\end{eqnarray}
where the third equality uses the assumption that $h(a)=0$ and in the last equality
$A$ and $B$ satisfy
\begin{eqnarray}
\frac{2}{\sigma^2}\frac{1}{\lambda_1+\lambda_2}\frac{1}{\lambda_1}e^{-\lambda_1 a} A = \lambda_1 A_1,
\quad
\frac{2}{\sigma^2}\frac{1}{\lambda_1+\lambda_2} \frac{1}{\lambda_2}e^{\lambda_2 a} B = -\lambda_2B_1.
\end{eqnarray}
The following theorem characterizes optimal parameters $(d^*,D^*,U^*,u^*)$ and parameters
$A^*$ and $B^*$ in \eqref{eq:g2-discounted} via solution $g=g_{A,B}$.
\begin{figure}[t]
  \centering
  \includegraphics[width=10cm]{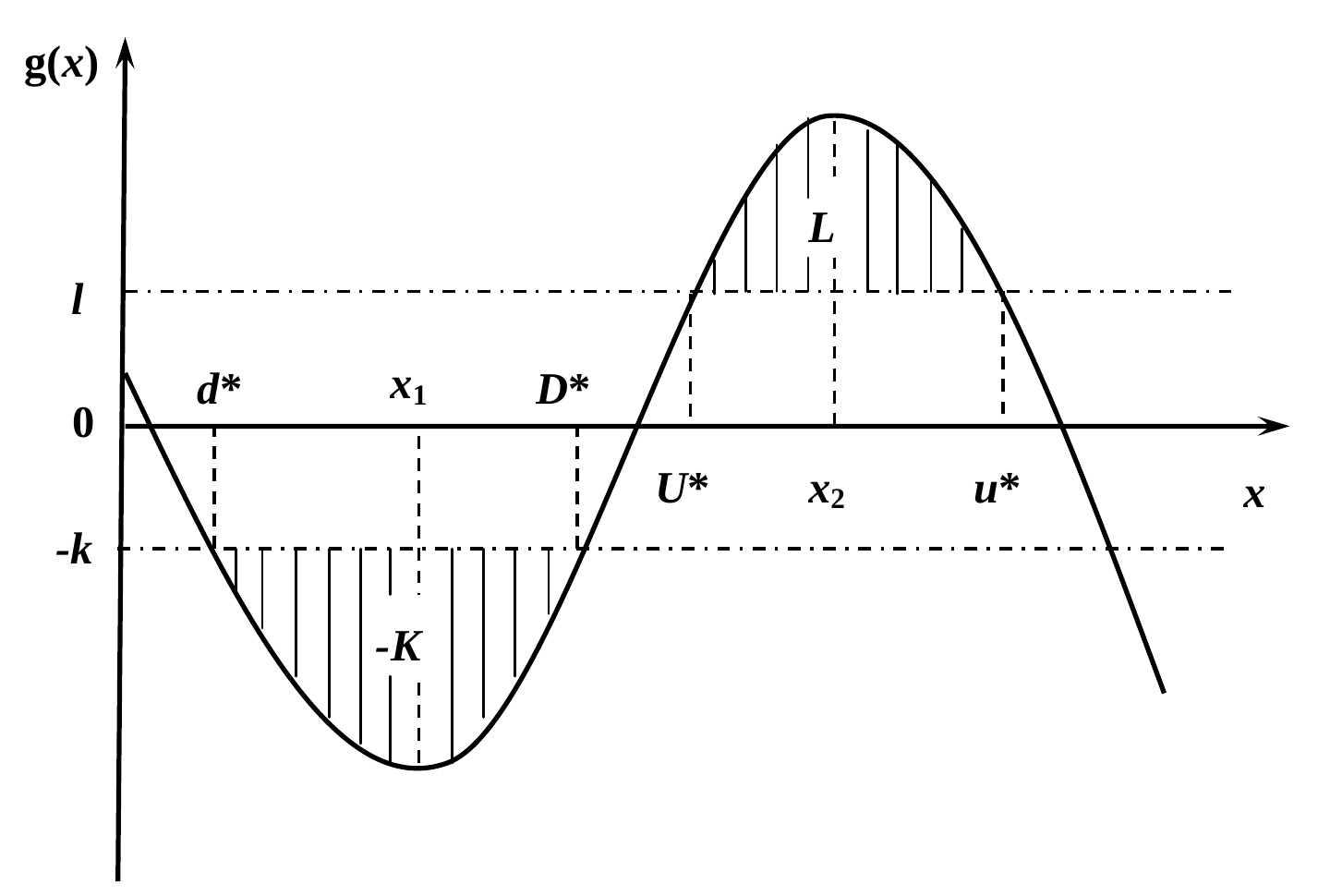}
\caption{ There exist $x_1<x_2$ such that  the function $g$  decreases in
  $(-\infty, x_1)$, increases in $(x_1, x_2)$, and deceases again in $(x_2,
  \infty)$. Parameters $d^*$, $D^*$, $U^*$ and $u^*$ are determined by
$g(d^*)=g(D^*)=-k$, $g(U^*)=g(u^*)=\ell$, the shaded area between
$U^*$ and $u^*$ is $L$, and the shaded area between $d^*$ and $D^*$ is
$K$. In the interval $[d^*, u^*]$, $g$ is the derivative of the
relative value function associated with the control band policy
$\{d^*, D^*, U^*, u^*\}$. }
\label{fig:g1}
\end{figure}
\begin{theorem}
\label{thm:optimalParameters-discounted}
 Assume that the holding cost function $h$ satisfies Assumption \ref{assumption:h-discounted}.
There exist unique $A^*$, $B^*$, $d^*$, $D^*$, $U^*$, $u^*$ with
\begin{equation}
  \label{eq:dlessDlessUlessuDiscount}
    d^*<x_1<D^* < U^* < x_2 <u^*.
\end{equation}
such that $g(x)=g_{A, B}(x)$ in~(\ref{eq:g2-discounted}) satisfies
\begin{eqnarray}
&& \int_d^D \Bigl[ g(x) + k \Bigr] dx=-K,  \label{eq:discountdD}\\
&& \int_U^u \Bigl[g(x) -\ell \Bigr]dx =L, \label{eq:discountUu}\\
&& g(d)=-k,   \label{eq:discountd} \\
&& g(D)=-k, \label{eq:discountD}\\
&& g(U)=\ell,  \label{eq:discountU} \\
&& g(u)=\ell.  \label{eq:discountu}
\end{eqnarray}
Furthermore, $g$ has a local minimum at $x_1<a$ and a local maximum at
$x_2>a$. The function $g$ is strictly decreasing on $(-\infty, x_1)$, strictly increasing
on $(x_1, x_2)$ and strictly decreasing again on $(x_2, \infty)$.
\end{theorem}
If $g$ satisfies all conditions (\ref{eq:gPoisson-discounted}),
(\ref{eq:discountdD})-(\ref{eq:discountu}) in Theorem~\ref{thm:optimalParameters-discounted},
$V(x)$ in \eqref{eq:V-discounted}  clearly satisfies all conditions
\eqref{eq:Poisson equ-discounted}-\eqref{eq:V(u)-V(U)-discounted} and \eqref{eq:optimalconditions1-discounted}-\eqref{eq:optimalconditions2-discounted}.
The proof of Theorem~\ref{thm:optimalParameters-discounted} is long, and we
defer it to Section \ref{sec:optimal-control-band-discounted}.

\begin{theorem}
\label{thm:optimal-discounted}
  Assume that the holding cost function $h$ satisfies Assumption
  \ref{assumption:h-discounted}. Let $d^*<D^*<U^*<u^*$, along with constants $A^*$ and
  $B^*$, be the unique solution in
  Theorem~\ref{thm:optimalParameters-discounted}. Then the control band policy
  $\varphi^*=\{d^*, D^*, U^*, u^*\}$ is optimal among all
  non-anticipating policies.
\end{theorem}
\begin{proof}
Let
\begin{eqnarray}
\label{eq:barg-discounted}
  \bar{g}(x) =
  \left\{
  \begin{array}{ll}
   -k & \text{for } x\in(-\infty,d^*],\\
   g_{A^*,B^*}(x) &\text{for } x\in (d^*,u^*),\\
   \ell & \text{for } x\in[u^*,\infty),
   \end{array}
  \right.
\end{eqnarray}
and
\begin{eqnarray}
\label{eq:barV*-discounted}
  \bar{V}(x) =
  \left\{
  \begin{array}{ll}
   V(D^*)+K+k(D^*-x) & \text{for } x\in(-\infty,d^*),\\
   V(x) &\text{for } x\in (d^*,u^*),\\
   V(U^*)+L+\ell(x-U^*) & \text{for } x\in(u^*,\infty),
   \end{array}
  \right.
\end{eqnarray}
with
\begin{eqnarray}
\label{eq:V*-discountd}
V(x)=A_1^* e^{\lambda_1 x } + B_1^* e^{-\lambda_2 x} + V_0(x),
\end{eqnarray}
where $\frac{2}{\sigma^2}\frac{1}{\lambda_1+\lambda_2}\frac{1}{\lambda_1}e^{-\lambda_1 a} A^* = \lambda_1 A_1^*$,
$\frac{2}{\sigma^2}\frac{1}{\lambda_1+\lambda_2} \frac{1}{\lambda_2}e^{\lambda_2 a} B^* = -\lambda_2B_1^*$ and
$V_0(x)$ is given by \eqref{eq:V0-discounted}. Therefore,
\begin{eqnarray}
\bar{V}'(x)=\bar{g}_{A^*,B^*}(x) \quad \text{for } x\in\R.
\end{eqnarray}

We now show that $\bar{V}$ satisfies all the
conditions in Theorem \ref{thm:lowerbound-discounted}. Thus, Theorem
\ref{thm:lowerbound-discounted} shows that the expected total discounted cost under any
feasible policy is at least $\bar{V}(x)$.
Since $\bar{V}(x)$ is the expected total discounted cost under the control band policy $\varphi^*$ with starting point $x$,
$\bar{V}(x)$ is the optimal cost and the control band policy $\varphi^*$ is optimal among all feasible policies.

First,  $\bar{V}(x)$ is in $C^2((d^*, u^*))$. Condition \eqref{eq:discountdD}
implies
\begin{equation}
  \label{eq:boundaryStarK}
 V(d^*)-V(D^*)=-\int_{d^*}^{D^*}g_{A^*,B^*}(x)dx= K  + k(D^*-d^*)
\end{equation}
and \eqref{eq:discountUu} implies
\begin{displaymath}
  V(u^*)-V(U^*) =\int_{U^*}^{u^*}g_{A^*,B^*}(x)dx= L  + \ell(u^*-U^*).
\end{displaymath}
\eqref{eq:V*-discountd} implies that $V$ satisfies
\begin{displaymath}
 \Gamma V(x) -\beta V(x)+ h(x) =0, \text{ for }  x\in [d^*, u^*].
\end{displaymath}
By Theorem \ref{thm:control band-discounted}, $\bar{V}$ defined in \eqref{eq:barV*-discounted} must be the discounted
 cost under control band policy $\varphi^*$.

Now, we show that $\bar{V}(x)$ satisfies the rest of conditions in Theorem
\ref{thm:lowerbound-discounted}. Conditions \eqref{eq:discountd} and \eqref{eq:discountu}
imply that truncated function $\bar{V}'(x)$ is continuous in $\R$.
Therefore, $\bar{V}\in \mathbb{C}^1(\R)$. Clearly, $\bar{V}'(x)=-k$ for $x\in
(-\infty,d^*]$ and  $\bar{V}'(x)=\ell$ for $x\in [u^*, \infty)$. Let
\begin{displaymath}
  M = \sup_{x\in [d^*, u^*]}\abs{g_{A^*,B^*}(x)}.
\end{displaymath}
We have $\abs{\bar{V}'(x)}\le M$ for all $x\in \R$.
Because
\begin{displaymath}
 \Gamma \bar{V}-\beta \bar{V}(x)+h(x) =\Gamma V-\beta V(x)+h(x) = 0 \quad \text{ for } x\in [d^*, u^*].
\end{displaymath}
In particular
\begin{displaymath}
\Gamma \bar{V}(d^*)-\beta \bar{V}(d^*)+h(d^*)=0
\end{displaymath}
and
\begin{displaymath}
\Gamma \bar{V}(u^*)-\beta \bar{V}(u^*)+h(u^*)=0.
\end{displaymath}
It follows from  part (a) and part (b) of
Lemma~\ref{lem:optimalParameter-discounted} that $d^*<x_1<a<x_2<u^*$,
$\bar{V}''(d^*)=V''(d^*)=g'(d^*)\le 0$  and $\bar{V}''(u^*)=V''(u^*)=g'(u^*)\le 0$ (see Figure \ref{fig:g1}).
Thus, we have $\mu \bar{V}'(d^*)-\beta \bar{V}(d^*)+h(d^*)\ge 0$
and $\mu \bar{V}'(u^*)-\beta \bar{V}(u^*)+h(u^*)\ge0$. Now,
for $x<d^*$, $\Gamma \bar{V}(x)-\beta \bar{V}(x)+h(x)=\mu (-k)-\beta (\bar{V}(d^*)-k(x-d^*))+h(x)\ge \mu \bar{V}'(d^*)-\beta \bar{V}(d^*)+h(d^*)\ge0$.
Similarly, for $x>u^*$, $\Gamma \bar{V}(x)-\beta \bar{V}(x)+h(x)=\mu(\ell)-\beta(\bar{V}(u^*)+\ell(x-u^*))+h(x)\ge\mu \bar{V}'(u^*)-\beta \bar{V}(u^*)+h(u^*)\ge 0$.

Now we verify that $\bar{V}$ satisfies \eqref{eq:lbK-discounted}.
Let $x, y\in \R$ with $y<x$.
Then,
\begin{eqnarray*}
  \bar{V}(x)-\bar{V}(y) + k(x-y) &= & \int_y^x[\bar g(z)+k]dz \\
    &\ge &  \int_{(y\vee d^*)\wedge D^*}^{(x\wedge D^*)\vee d^*}[\bar g(z)+k]dz \\
    &\ge &  \int_{d^*}^{D^*}[\bar g(z)+k]dz \\
  &=& -K,
\end{eqnarray*}
where the first inequality follows from $\bar g(z)=-k$ for $z\le
d^*$ and  $\bar g(z)=g(z)\ge -k$ for $D^*< z<
u^*$ and $\bar g(z)=\ell \ge -k$ for $z\ge u^*$,
 and the second inequality follows from the fact that $\bar
 g(z)=g(z)\le -k$ for
$z\in [d^*, D^*]$; see, Figure \ref{fig:g1}. Thus \eqref{eq:lbK-discounted} is
proved.

It remains to verify that $\bar{V}$ satisfies \eqref{eq:lbL-discounted}.
For $x, y\in \R$ with $y>x$.
\begin{eqnarray*}
  \bar{V}(y)-\bar{V}(x) -\ell (y-x) &= & \int_x^y[\bar g(z)-\ell]dz \\
    &\le &  \int_{(x\vee U^*)\wedge u^*}^{(y\wedge u^*)\vee U^*}[\bar g(z)-\ell]dz \\
    &\le &  \int_{U^*}^{u^*}[\bar g(z)-\ell]dz \\
  &=& L,
\end{eqnarray*}
proving \eqref{eq:lbL-discounted}.
\end{proof}

\section{Optimal Control Band Parameters}
\label{sec:optimal-control-band-discounted}

This section is devoted to the proof of Theorem \ref{thm:optimalParameters-discounted}.
We separate the proof into a series of lemmas.

Since $h$ is convex, one has $h'(x)\le h'(y)$ whenever
the derivatives at $x< y$ exist. It follows that $\lim_{x\uparrow
  a}h'(x)$ and $\lim_{x\downarrow a}h'(x)$ exist.
Define
\begin{eqnarray*}
h'(a-)=\lim_{x\uparrow a}h'(x) \quad \text{ and } \quad
h'(a+)=\lim_{x\downarrow a}h'(x).
\end{eqnarray*}
We have $h'(a-)\le h'(a+)$. Recall the function $g$ in
(\ref{eq:g2-discounted}). Using the integration by parts, one has
\begin{eqnarray}
g(x) &=&\left\{
\begin{array}{ll}
\frac{2}{\sigma^2}\frac{1}{\lambda_1+\lambda_2}\Big[\frac{1}{\lambda_1}
 \Big(A-h'(a-)+\int_x^a e^{-\lambda_1 (y-a)}h''(y)dy\Big)e^{\lambda_1 (x-a)}
\\
\quad+\frac{1}{\lambda_2}\Big(B-h'(a-)+\int_x^a e^{\lambda_2 (y-a)}h''(y)dy\Big) e^{-\lambda_2 (x-a)}
  \Big]+\frac{1}{\beta}h'(x) & \text{for $x<a$},\\
 \frac{2}{\sigma^2}\frac{1}{\lambda_1+\lambda_2}\biggl[\frac{1}{\lambda_1}
 \Big(A-h'(a+)-\int_a^x e^{-\lambda_1 (y-a)}h''(y)dy\Big)e^{\lambda_1 (x-a)}
\\
\quad+\frac{1}{\lambda_2}\Big(B-h'(a+)-\int_a^x e^{\lambda_2 (y-a)}h''(y)dy\Big) e^{-\lambda_2 (x-a)}
  \Big]+\frac{1}{\beta}h'(x) & \text{for $x>a$}.
\end{array}
\right.
\label{eq:g3-discounted}
\end{eqnarray}
It follows that
\begin{eqnarray}
\label{eq:g'-discounted}
g'(x)
=\left\{
\begin{array}{ll}
\frac{2}{\sigma^2}\frac{1}{\lambda_1+\lambda_2}\Big[
 \Big(A-h'(a-)+\int_x^a e^{-\lambda_1 (y-a)}h''(y)dy\Big)e^{\lambda_1 (x-a)}
\\
\quad -\Big(B-h'(a-)+\int_x^a e^{\lambda_2 (y-a)}h''(y)dy\Big) e^{-\lambda_2 (x-a)}
  \Big] & \text{for $x<a$},\\
 \frac{2}{\sigma^2}\frac{1}{\lambda_1+\lambda_2}\Big[
 \Big(A-h'(a+)-\int_a^x e^{-\lambda_1 (y-a)}h''(y)dy\Big)e^{\lambda_1 (x-a)}
\\
\quad -\Bigl(B-h'(a+)-\int_a^x e^{\lambda_2 (y-a)}h''(y)dy\Big) e^{-\lambda_2 (x-a)}
  \Big] & \text{for $x>a$}
\end{array}
\right.
\end{eqnarray}
and
\begin{eqnarray}
\label{eq:g''-discounted}
g''(x)=\left\{
\begin{array}{ll}
\frac{2}{\sigma^2}\frac{1}{\lambda_1+\lambda_2}\Big[
 \Big(A-h'(a-)+\int_x^a e^{-\lambda_1 (y-a)}h''(y)dy\Big)\lambda_1e^{\lambda_1 (x-a)}
\\
\quad +\Big(B-h'(a-)+\int_x^a e^{\lambda_2 (y-a)}h''(y)dy\Big) \lambda_2e^{-\lambda_2 (x-a)}
  \Big] & \text{for $x<a$},\\
 \frac{2}{\sigma^2}\frac{1}{\lambda_1+\lambda_2}\Big[
 \Big(A-h'(a+)-\int_a^x e^{-\lambda_1 (y-a)}h''(y)dy\Big)\lambda_1e^{\lambda_1 (x-a)}
\\
\quad +\Big(B-h'(a+)-\int_a^x e^{\lambda_2 (y-a)}h''(y)dy\Big) \lambda_2e^{-\lambda_2 (x-a)}
  \Big] & \text{for $x>a$}.
\end{array}
\right.
\end{eqnarray}
Define
\begin{eqnarray*}
\overline{A}=h'(a+)+\int_a^{+\infty} e^{-\lambda_1 (y-a)}h''(y)dy,
\quad
\underline{B}=h'(a-)-\int_{-\infty}^a e^{\lambda_2 (y-a)}h''(y)dy.
\end{eqnarray*}
We have the following lemma.
\begin{lemma}
  \label{lem:nonemptyregion}
Assume that $h$ satisfies Assumption \ref{assumption:h-discounted}, then
\begin{eqnarray}
&&h'(a-) <\overline{A},\label{eq:righSide}\\
&&\underline{B} <h'(a+). \label{eq:leftSide}
\end{eqnarray}
\end{lemma}
\begin{proof}
Assumption \ref{assumption:h-discounted} (c) says that
$h'(x)\leq 0$ for $x<a$ and $h'(x)\geq 0$ for $x>a$,
we have
\begin{eqnarray}
\label{eq:h'(a-)<h'(a+)-discounted}
  h'(a-)=\lim_{x\uparrow a}h'(x)\le 0 \le\lim_{x\downarrow a}h'(x)=h'(a+).
\end{eqnarray}
If $\int_a^{+\infty} e^{-\lambda_1 (y-a)}h''(y)dy>0$, then $~(\ref{eq:righSide})$ clearly holds.
Now assume that
\begin{displaymath}
  \int_a^{+\infty} e^{-\lambda_1 (y-a)}h''(y)dy=0.
\end{displaymath}
Because $h''(x)\ge 0$  and $h''(x)$ is assumed to be continuous on
$(a, \infty)$, we
have $h''(x)=0$ for $x>a$. Therefore, $h$ must be linear in $x>a$.
This fact and \eqref{eq:hlimit} imply that
\begin{displaymath}
h'(a+)=\lim_{x\downarrow a}h'(x)>0,
\end{displaymath}
which, together with \eqref{eq:h'(a-)<h'(a+)-discounted} yields
\eqref{eq:righSide}.

Similarly we can prove \eqref{eq:leftSide}.
\end{proof}

\begin{figure}[t]
  \centering
  \includegraphics[width=12cm]{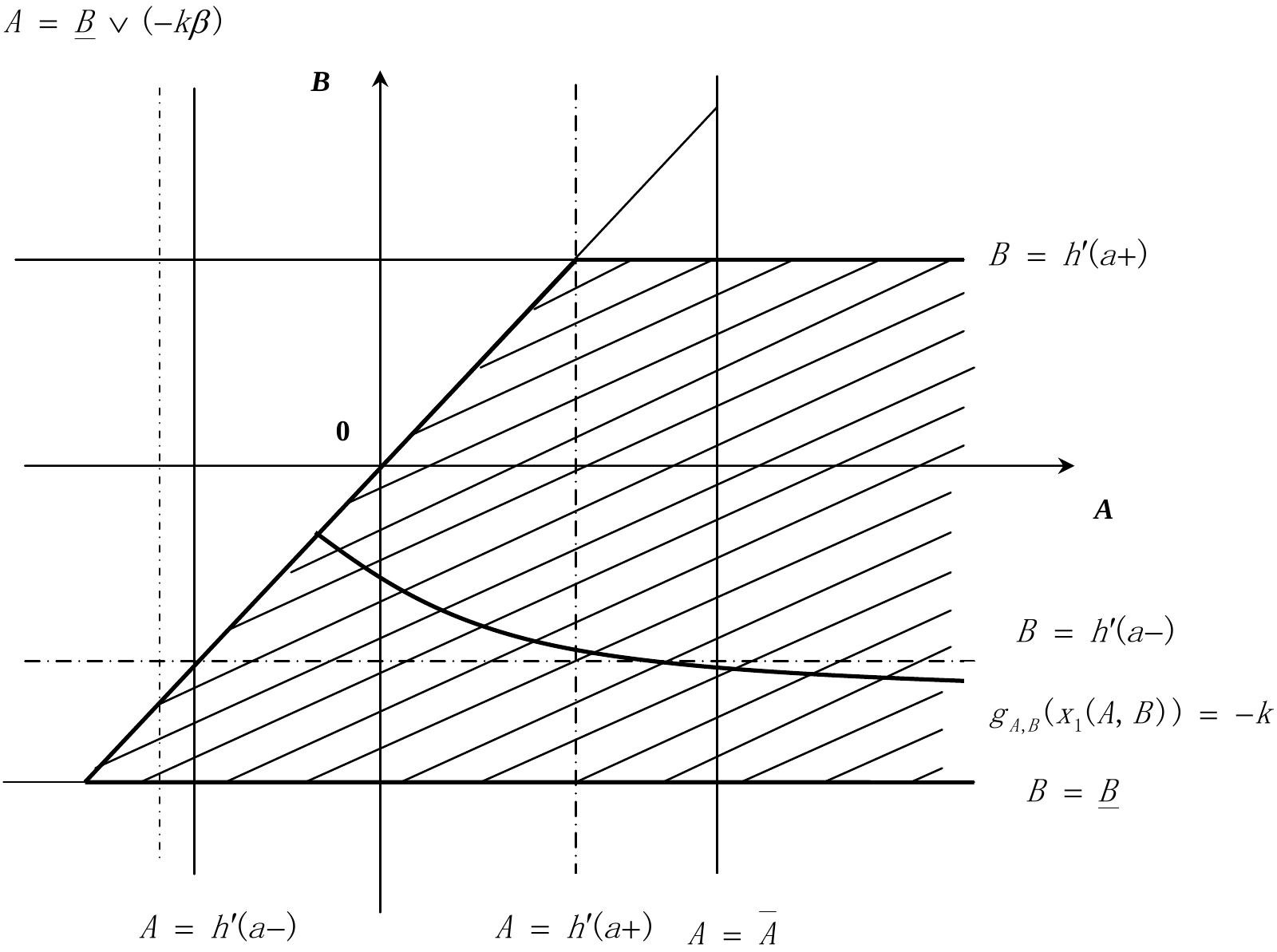}
\caption{(a) The shaded region is the set of $(A,B)$ that satisfies
  (\ref{eq:ABcondition1-discounted}) and
  (\ref{eq:ABcondition2-discounted}). The unique minimum
  $x_1=x_1(A,B)\in (-\infty, a)$ is well defined for all $(A,B)$ in
  this region.
 \newline
(b) For each $A \in \bigl(\underline{B}\vee(-k\beta), +\infty\bigr)$,
there exists a unique $\overline{B}(A) \in \bigl(\underline{B},
A\wedge h'(a+)\bigr)$
such that $g_{A,\overline{B}(A)}(x_1(A,\overline{B}(A)))=-k$. The
curve $B=\overline{B}(A)$ is decreasing.}
\label{fig:dis1a}
\end{figure}

\begin{figure}[t]
  \centering
  \includegraphics[width=12cm]{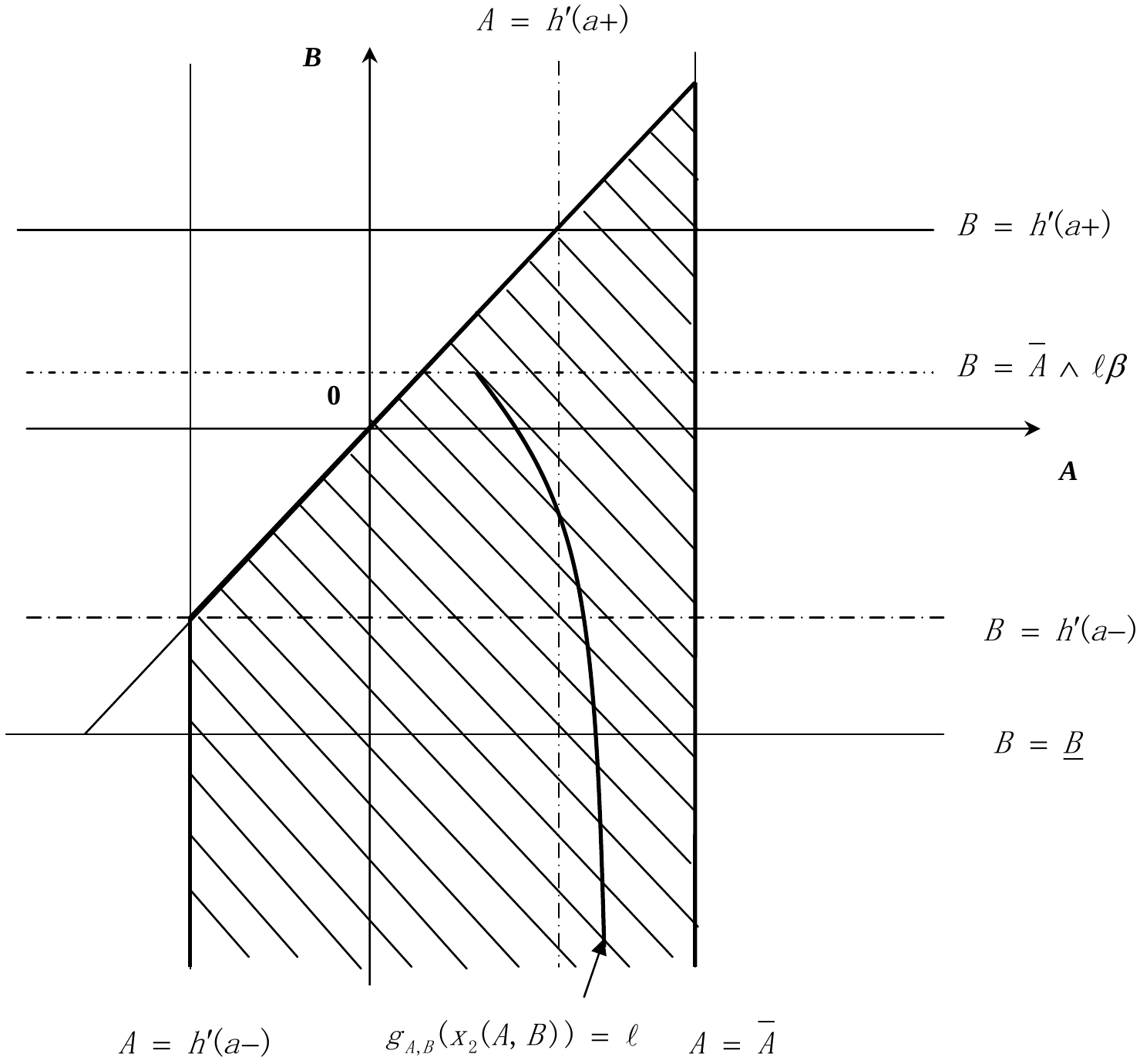}
\caption{(a) The shaded region is the set of all $(A, B)$ that
  satisfies (\ref{eq:ABcondition2-discounted}) and
  (\ref{eq:ABcondition3-discounted}). The unique maximum
  $x_2=x_2(A, B)\in(a,\infty)$ is well defined for all $(A,B)$ in the
  region. \newline
(b) For each $B\in \bigl(-\infty, \overline{A}\wedge \ell
\beta\bigr)$, there exists a unique
$\underline{A}(B)\in \bigl(B\vee h'(a-), \overline{A}\bigr)$ such that
$g_{\underline{A}(B),B}(x_2(\underline{A}(B),B))=\ell$. The curve
$A=\underline{A}(B)$ is decreasing.}
\label{fig:dis1b}
\end{figure}

\begin{lemma}
  \label{lem:optimalParameter-discounted}
  (a) For each  $B$ satisfying
  \begin{eqnarray}
&&\underline{B}<B<h'(a+), \label{eq:ABcondition1-discounted}
\end{eqnarray}
and each $A$ satisfying
\begin{equation}
 B< A,\label{eq:ABcondition2-discounted}
\end{equation}
$g_{A, B}(x)$ attains a unique minimum in $(-\infty, a)$ at $x_1=x_1(A,B)\in (-\infty, a)$.

For each  $A$ satisfying
\begin{eqnarray}
&&h'(a-)<A<\overline{A},\label{eq:ABcondition3-discounted}
\end{eqnarray}
and each $B$ satisfying (\ref{eq:ABcondition2-discounted}),
$g_{A, B}(x)$ attains a unique maximum in $(a, \infty)$ at $x_2=x_2(A,B) \in (a, \infty)$.\\
(b) For each fixed $A$ and $B$ satisfying
\eqref{eq:ABcondition1-discounted}-\eqref{eq:ABcondition2-discounted},
the local minimizer $x_1=x_1(A,B)$ is the unique solution in
$(-\infty, a)$ to
  \begin{eqnarray}
    \label{eq:x1ABequation-discounted}
    &&\Bigl(A-h'(a-)+\int_{x_1}^a e^{-\lambda_1 (y-a)}h''(y)dy\Bigr)e^{\lambda_1 (x_1-a)}\\
    &&\quad=\Bigl(B-h'(a-)+\int_{x_1}^a e^{\lambda_2 (y-a)}h''(y)dy\Bigr) e^{-\lambda_2 (x_1-a)}.\nonumber
  \end{eqnarray}
Furthermore, $g_{A, B}'(x)<0$ for $x\in (-\infty, x_1(A,B))$, $g_{A, B}'(x)>0$ for $x\in (x_1(A,B), a)$, and
\begin{eqnarray}
&&\lim_{x\downarrow-\infty} g_{A,B}(x)=+\infty,\label{eq:limg1-discounted}\\
&& g_{A,B} ''(x_1(A,B))>0.\label{eq:g''(x1)>0-discounted}
\end{eqnarray}

 For each fixed $A$ and $B$ satisfying \eqref{eq:ABcondition2-discounted}-\eqref{eq:ABcondition3-discounted},
 the local maximizer $x_2=x_2(A,B)$ is the unique solution in
 $(a,\infty)$ to
\begin{eqnarray}
\label{eq:x2ABequation-discounted}
&&\Bigl(A-h'(a+)-\int_a^{x_2} e^{-\lambda_1 (y-a)}h''(y)dy\Bigr)e^{\lambda_1 (x_2-a)}\\
&&\quad=\Bigl(B-h'(a+)-\int_a^{x_2} e^{\lambda_2 (y-a)}h''(y)dy\Bigr) e^{-\lambda_2 (x_2-a)}.\nonumber
\end{eqnarray}
Furthermore, $g_{A, B}'(x)>0$ for $x\in (a, x_2(A,B))$, $g_{A, B}'(x)<0$ for $x\in(x_2(A,B), \infty)$,
and
\begin{eqnarray}
&&\lim_{x\uparrow \infty} g_{A,B}(x)=-\infty, \label{eq:limg2-discounted}\\
&& g_{A,B} ''(x_2(A,B))<0.\label{eq:g''(x2)<0-discounted}
\end{eqnarray}
\end{lemma}
\begin{remark}
(a) The set of  $(A, B)$ that satisfies
(\ref{eq:ABcondition1-discounted}) and
(\ref{eq:ABcondition2-discounted}) is the shaded region in Figure
\ref{fig:dis1a}. The set of  $(A, B)$ that satisfies
(\ref{eq:ABcondition2-discounted}) and
(\ref{eq:ABcondition3-discounted}) is the shaded region in Figure
\ref{fig:dis1b}. \\
(b)
Note that
\begin{eqnarray}
&&\bigl(A-\lambda_1\int_a^{x_2(A,B)} e^{-\lambda_1 (y-a)}h'(y)dy\bigr)e^{\lambda_1 (x_2(A,B)-a)}\nonumber\\
&&\quad =\bigl(A-h'(a+)-\int_a^{x_2(A,B)} e^{-\lambda_1 (y-a)}h''(y)dy\bigr)e^{\lambda_1 (x_2(A,B)-a)}+h'(x_2(A,B))\nonumber\\
&&\quad =\bigl(B-h'(a+)-\int_a^{x_2(A,B)} e^{\lambda_2 (y-a)}h''(y)dy\bigr)e^{-\lambda_2 (x_2(A,B)-a)}+h'(x_2(A,B))\nonumber\\
&&\quad =\bigl(B+\lambda_2\int_a^{x_2(A,B)} e^{\lambda_2 (y-a)}h'(y)dy\bigr)e^{-\lambda_2 (x_2(A,B)-a)},\label{eq:first=second-discounted}
\end{eqnarray}
where the first and third equalities follow from integration by
parts, and the second is due to
the definition of $x_2(A,B)$ in \eqref{eq:x2ABequation-discounted}.
This provide an alternative characterization of $x_2(A,B)$ in
\eqref{eq:x2ABequation-discounted}. Similarly, $x_1(A,B)$ has an
alternative characterization.
\end{remark}
\begin{proof}
We only prove the existence of $x_1$ and the properties of $g(x)$ in $x\in(-\infty,a)$.
The proof for the existence of $x_2$ and the properties of $g(x)$ in $x\in(a,\infty)$ is similar,
and it is omitted.

In order to prove the existence of $x_1$, we divide $B\in\bigl(\underline{B}, h'(a+)\bigr)$ into two cases:
$B\in\bigl(\underline{B}, h'(a-)\bigr]$
and $B\in\bigl(h'(a-), h'(a+)\bigr)$.

\textbf{Case 1.} $B\in\bigl(\underline{B}, h'(a-)\bigr]$.

Note that $h$ is convex, we have $h''(x)\geq0$ for all $x\in\R$ except $x=a$.
Therefore, $\int_{x}^a e^{\lambda_2 (y-a)}h''(y)dy\geq 0$ is decreasing in $x\in(-\infty,a)$.
Then for fixed $B\in\bigl(\underline{B}, h'(a-)\bigr]$,
there exists an $x'$ with $x'\in(-\infty,a]$ such that
\begin{eqnarray}
\label{eq:B=-discounted}
B=h'(a-)-\int_{x'}^a e^{\lambda_2 (y-a)}h''(y)dy.
\end{eqnarray}
We are going to prove that $g'(x)$ is strictly increasing in $x\in(-\infty,x')$
and
\begin{eqnarray}
&& \lim_{x\downarrow -\infty} g'(x)=-\infty, \label{eq:gprimeminusinfinty}\\
&& \lim_{x\uparrow x'} g'(x)>0, \label{eq:gprimexprime} \\
&&  g'(x)>0 \quad \text{ for } x\in (x', a). \label{eq:gprimepositive}
\end{eqnarray}
Since $g'(x)$ is continuous and strictly increasing in $x\in (-\infty, x')$,
(\ref{eq:gprimeminusinfinty}) and (\ref{eq:gprimexprime}) imply that
there exists a unique $x_1$ with $x_1\in(-\infty,x')$ such that
\begin{eqnarray*}
g'(x)
\left\{
\begin{array}{ll}
<0, & x<x_1,\\
=0, & x=x_1,\\
>0, & x_1<x<x'.
\end{array}
\right.
\end{eqnarray*}
Combining this  with (\ref{eq:gprimepositive}), we have
\begin{eqnarray*}
g'(x)
\left\{
\begin{array}{ll}
<0, & x<x_1,\\
=0, & x=x_1,\\
>0, & x_1<x<a,
\end{array}
\right.
\end{eqnarray*}
from which one proves the existence of $x_1$ and properties of $g(x)$
in $(-\infty,a)$.

It remains to prove that $g'(x)$ is strictly increasing in $x\in(-\infty,x')$,
and that (\ref{eq:g''(x1)>0-discounted}), and
(\ref{eq:gprimeminusinfinty})-(\ref{eq:gprimepositive}) hold.
We first prove that $g'(x)$ is strictly increasing in $x\in(-\infty,x')$.
For $x\in(-\infty,x')$,
\begin{eqnarray*}
B-h'(a-)+\int_x^a e^{\lambda_2 (y-a)}h''(y)dy\geq0
\end{eqnarray*}
and
\begin{eqnarray}
A-h'(a-)+\int_x^a e^{-\lambda_1 (y-a)}h''(y)dy
&>&B-h'(a-)+\int_x^a e^{-\lambda_1 (y-a)}h''(y)dy\nonumber\\
&\geq&B-h'(a-)+\int_x^a e^{\lambda_2 (y-a)}h''(y)dy\nonumber\\
&\geq&0,\label{eq:A>0-discounted}
\end{eqnarray}
where the first inequality is due to \eqref{eq:ABcondition2-discounted}.
Using \eqref{eq:g''-discounted}, we further have that for $x\in(-\infty,x')$,
\begin{eqnarray}
g''(x)&=&\frac{2}{\sigma^2}\frac{1}{\lambda_1+\lambda_2}\Bigl[
 \Bigl(A-h'(a-)+\int_x^a e^{-\lambda_1 (y-a)}h''(y)dy\Bigr)\lambda_1e^{\lambda_1 (x-a)}\nonumber\\
&&+\Bigl(B-h'(a-)+\int_x^a e^{\lambda_2 (y-a)}h''(y)dy\Bigr) \lambda_2e^{-\lambda_2 (x-a)}
  \Bigr]\nonumber\\
&>&0.\label{eq:g''>0-discounted}
\end{eqnarray}
This proves $g'(x)$ is strictly increasing in $(-\infty, x')$.

To see (\ref{eq:gprimeminusinfinty}), it follows from \eqref{eq:g'-discounted} that
\begin{eqnarray*}
\lim_{x\downarrow-\infty}\frac{g'(x)}{e^{-\lambda_2(x-a)}}&=&\lim_{x\downarrow
-\infty} \frac{2}{\sigma^2}\frac{1}{\lambda_1+\lambda_2}\Bigl[
-\Bigl(B-h'(a-)+\int_x^a e^{\lambda_2 (y-a)}h''(y)dy\Bigr) \\
&& +
 \Bigl(A-h'(a-)+\int_x^a e^{-\lambda_1
h'(x)(y-a)}h''(y)dy\Bigr)e^{(\lambda_1+\lambda_2) (x-a)}
  \Bigr].
\end{eqnarray*}
To  evaluate this limit, we first have
\begin{eqnarray}
\lefteqn{\lim_{x\downarrow -\infty}\Bigl(A-h'(a-)+\int_x^a e^{-\lambda_1
h'(x)(y-a)}h''(y)dy\Bigr)e^{(\lambda_1+\lambda_2) (x-a)}}\nonumber\\
&=&\lim_{x\downarrow -\infty}\int_x^a e^{-\lambda_1
h'(x)(y-a)}h''(y)dy\cdot e^{(\lambda_1+\lambda_2) (x-a)}\nonumber\\
&=& 0, \label{eq:lim-first-discounted}
\end{eqnarray}
where the last equality follows from (\ref{eq:limint-discounted}).
Next,
\begin{eqnarray}
&&\lim_{x\downarrow -\infty}\Bigl(B-h'(a-)+\int_x^a e^{\lambda_2
h'(x)(y-a)}h''(y)dy\Bigr) \nonumber\\
&&\quad =B-\underline{B}.\label{eq:lim-second-discounted}
\end{eqnarray}
Because $B-\underline{B}>0$, \eqref{eq:lim-first-discounted} and \eqref{eq:lim-second-discounted} imply that
\begin{eqnarray}
\label{eq:limg'-discounted}
\lim_{x\downarrow -\infty} g'(x)=-\infty.
\end{eqnarray}

To see (\ref{eq:gprimexprime}),  it follows from
\eqref{eq:g'-discounted}  that
\begin{eqnarray}
\lim_{x\uparrow x'} g'(x)&=&\frac{2}{\sigma^2}\frac{1}{\lambda_1+\lambda_2}\Bigl[
 \bigl(A-h'(a-)+\int_{x'}^a e^{-\lambda_1 (y-a)}h''(y)dy\bigr)e^{\lambda_1 (x'-a)}\nonumber\\
&&\quad \quad \quad \quad \quad -\bigl(B-h'(a-)+\int_{x'}^a e^{\lambda_2 (y-a)}h''(y)dy\bigr) e^{-\lambda_2 (x'-a)}
  \Bigr]\nonumber\\
&=&\frac{2}{\sigma^2}\frac{1}{\lambda_1+\lambda_2}\Bigl[
 \bigl(A-h'(a-)+\int_{x'}^a e^{-\lambda_1 (y-a)}h''(y)dy\bigr)e^{\lambda_1 (x'-a)}
  \Bigr]\nonumber\\
&\geq&\frac{2}{\sigma^2}\frac{1}{\lambda_1+\lambda_2}\Bigl[
 \bigl(A-h'(a-)+\int_{x'}^a e^{\lambda_2 (y-a)}h''(y)dy\bigr)e^{\lambda_1 (x'-a)}
  \Bigr]\nonumber\\
&=&\frac{2}{\sigma^2}\frac{1}{\lambda_1+\lambda_2}(A-B)e^{\lambda_1 (x'-a)}\label{eq:limg'1-discounted}\\
&>&0,\nonumber
\end{eqnarray}
where the second and last equalities are due to \eqref{eq:B=-discounted} and
the last inequality is due to \eqref{eq:ABcondition2-discounted}.

To see (\ref{eq:gprimepositive}),
for $x\in[x',a)$, \eqref{eq:B=-discounted} implies that
$B-h'(a-)+\int_{x}^a e^{\lambda_2 (y-a)}h''(y)dy\leq0$,
which plus
\begin{eqnarray}
A-h'(a-)+\int_x^a e^{-\lambda_1 (y-a)}h''(y)dy
&>&B-h'(a-)+\int_x^a e^{-\lambda_1 (y-a)}h''(y)dy\nonumber\\
&\geq&B-h'(a-)+\int_x^a e^{\lambda_2 (y-a)}h''(y)dy\nonumber
\end{eqnarray}
imply that
\begin{eqnarray*}
g'(x)&=&\frac{2}{\sigma^2}\frac{1}{\lambda_1+\lambda_2}\Big[
 \bigl(A-h'(a-)+\int_{x}^a e^{-\lambda_1 (y-a)}h''(y)dy\bigr)e^{\lambda_1 (x-a)}\\
&& -\bigl(B-h'(a-)+\int_{x}^a e^{\lambda_2 (y-a)}h''(y)dy\bigr) e^{-\lambda_2 (x-a)}
  \Big]\\
&>& \frac{2}{\sigma^2}\frac{1}{\lambda_1+\lambda_2}\Big[\bigl(B-h'(a-)+\int_{x}^a e^{\lambda_2 (y-a)}h''(y)dy\bigr)
\bigl(e^{\lambda_1 (x-a)}-e^{-\lambda_2 (x-a)}\bigr)\Big]\\
&\geq&0.
\end{eqnarray*}

\textbf{Case 2.}  $B\in\bigl(h'(a-), h'(a+)\bigr)$

It is similar to prove \eqref{eq:g''>0-discounted}, \eqref{eq:limg'-discounted} and \eqref{eq:limg'1-discounted},
we have
\begin{eqnarray}
&&g''(x)>0 \quad\text{for $x\in(-\infty,a)$},\label{eq:g''>0-discounted1}\\
&&\lim_{x\downarrow -\infty} g'(x)=-\infty\label{eq:limg'2-discounted}
\end{eqnarray}
and
\begin{eqnarray}
\lim_{x\uparrow a} g'(x)=\frac{2}{\sigma^2}\frac{1}{\lambda_1+\lambda_2}(A-B)>0.\nonumber
\end{eqnarray}
Therefore, there exists a unique $x_1$ such that
\begin{eqnarray*}
g'(x)
\left\{
\begin{array}{ll}
<0, & x<x_1,\\
=0, & x=x_1,\\
>0, & x_1<x<a.
\end{array}
\right.
\end{eqnarray*}

Limit \eqref{eq:limg1-discounted} can immediately be obtained by
\eqref{eq:limg'-discounted} and \eqref{eq:limg'2-discounted}.
Inequalities
\eqref{eq:g''>0-discounted} and \eqref{eq:g''>0-discounted1} and the definition of $x_1$
easily imply \eqref{eq:g''(x1)>0-discounted}.
\end{proof}

\begin{lemma}
\label{lem:x_i-discounted}
Suppose $A$ and $B$ satisfy \eqref{eq:ABcondition1-discounted}-\eqref{eq:ABcondition2-discounted},
for fixed $B$, the local minimizer  $x_1(A,B)$ is continuous and strictly decreasing in $A$;
for fixed $A$, the local minimizer  $x_1(A,B)$ is continuous and strictly increasing in $B$.
Suppose $A$ and $B$ satisfy \eqref{eq:ABcondition2-discounted}-\eqref{eq:ABcondition3-discounted},
for fixed $B$, the local maximizer $x_2(A,B)$ is continuous and strictly increasing in $A$;
for fixed $A$, the local maximizer $x_2(A,B)$ is continuous and strictly decreasing in $B$.

Furthermore,
\begin{eqnarray}
&&\lim_{B\downarrow \underline{B}} x_1(A,B)=-\infty, \label{eq:x1ABlimit1-discounted}\\
&&\lim_{B\uparrow A} x_1(A,B)=a \quad \text{for $A<h'(a+)$}, \label{eq:x1ABlimit2-discounted}\\
&&\lim_{A\downarrow B} x_2(A,B)=a \quad \text{for $B>h'(a-)$},\label{eq:x2ABlimit1-discounted}\\
&&\lim_{A\uparrow \overline{A}} x_2(A,B)=\infty .\label{eq:x2ABlimit2-discounted}
\end{eqnarray}
\end{lemma}
\begin{proof}
The {\em Implicit Function Theorem} implies the continuity of $x_i(A,B)$, $i=1,2$.
Applying the {\em Implicit Function Theorem} to \eqref{eq:x1ABequation-discounted} and \eqref{eq:x2ABequation-discounted},
we have that
\begin{eqnarray}
\frac{\partial x_1(A,B)}{\partial A}&=&-\frac{2}{\sigma^2}\frac{1}{\lambda_1+\lambda_2}\frac{e^{\lambda_1 (x_1(A,B)-a)}}{g''(x_1(A,B))}<0,\label{eq:dx1/dA-discounted}\\
\frac{\partial x_1(A,B)}{\partial B}&=&\frac{2}{\sigma^2}\frac{1}{\lambda_1+\lambda_2}\frac{e^{-\lambda_2 (x_1(A,B)-a)}}{g''(x_1(A,B))}>0,\label{eq:dx1/dB-discounted}\\
\frac{\partial x_2(A,B)}{\partial A}&=&-\frac{2}{\sigma^2}\frac{1}{\lambda_1+\lambda_2}\frac{e^{\lambda_1 (x_2(A,B)-a)}}{g''(x_2(A,B))}>0,\label{eq:dx2/dA-discounted}\\
\frac{\partial x_2(A,B)}{\partial B}&=&\frac{2}{\sigma^2}\frac{1}{\lambda_1+\lambda_2}\frac{e^{-\lambda_2 (x_2(A,B)-a)}}{g''(x_2(A,B))}<0,\label{eq:dx2/dB-discounted}
\end{eqnarray}
where in obtaining \eqref{eq:dx1/dA-discounted} and
\eqref{eq:dx1/dB-discounted} we
have used $g''(x_1(A,B))>0$ in \eqref{eq:g''(x1)>0-discounted},
and in obtaining
\eqref{eq:dx2/dA-discounted} and \eqref{eq:dx2/dB-discounted} we have used
$g''(x_2(A,B))<0$ in \eqref{eq:g''(x2)<0-discounted}.

Fix $A$ satisfying $A<h'(a+)$, when $x_1\uparrow a$,
\eqref{eq:x1ABequation-discounted} gives that $B\uparrow A$.
From the monotonicity between $x_1$ and $B$, we must have
\eqref{eq:x1ABlimit2-discounted}.

We next prove \eqref{eq:x1ABlimit1-discounted}.
From \eqref{eq:x1ABequation-discounted}, we have
\begin{eqnarray}
&&\Bigl(A-h'(a-)+\int_{x_1}^a e^{-\lambda_1 (y-a)}h''(y)dy\Bigr)e^{(\lambda_1+\lambda_2) (x_1-a)}\nonumber\\
&&\quad=B-h'(a-)+\int_{x_1}^a e^{\lambda_2 (y-a)}h''(y)dy.\label{eq:x1ABequation1-discounted}
\end{eqnarray}
We will show that
\begin{eqnarray}
\label{eq:lim=0-discounted}
\lim_{x_1\downarrow -\infty} \Bigl(A-h'(a-)+\int_{x_1}^a e^{-\lambda_1 (y-a)}h''(y)dy\Bigr)e^{(\lambda_1+\lambda_2) (x_1-a)}=0.
\end{eqnarray}
This, together with \eqref{eq:x1ABequation1-discounted}, implies that
\begin{eqnarray*}
0&=&\lim_{x_1\downarrow -\infty}\Bigl(B-h'(a-)+\int_{x_1}^a e^{\lambda_2 (y-a)}h''(y)dy\Bigr)\\
&=&\lim_{x_1\downarrow -\infty}B -h'(a-)+\int_{-\infty}^a e^{\lambda_2
  (y-a)}h''(y)dy\\
&=&\lim_{x_1 \downarrow -\infty} B - \underline{B},
\end{eqnarray*}
from which one has that $B\downarrow \underline{B}$ when
$x_1\downarrow -\infty$.
Using the monotonicity between $x_1$ and $B$ (see
(\ref{eq:dx1/dB-discounted})), we must have
\eqref{eq:x1ABlimit1-discounted}.

It remains to prove \eqref{eq:lim=0-discounted}. To see this,
\begin{eqnarray*}
&&\lim_{x_1\downarrow -\infty} \Bigl(A-h'(a-)+\int_{x_1}^a e^{-\lambda_1 (y-a)}h''(y)dy\Bigr)e^{(\lambda_1+\lambda_2) (x_1-a)}\\
&&\quad=\lim_{x_1\downarrow -\infty} \int_{x_1}^a e^{-\lambda_1 (y-a)}h''(y)dy\cdot e^{(\lambda_1+\lambda_2) (x_1-a)}\\
&&\quad=0,
\end{eqnarray*}
where the last equality is due to \eqref{eq:limint-discounted}.
Therefore, we have proved \eqref{eq:lim=0-discounted}.

The proof for \eqref{eq:x2ABlimit1-discounted} and \eqref{eq:x2ABlimit2-discounted} is similar.
\end{proof}

\begin{lemma}
\label{lem:underlineAoverlineB-discounted}
(a) For each
\begin{equation}
  \label{eq:Brange}
B\in \bigl(-\infty, \overline{A}\wedge \ell \beta\bigr),
\end{equation}
there exists a
unique
\begin{displaymath}
\underline{A}(B)\in \bigl(B\vee h'(a-), \overline{A}\bigr)
\end{displaymath}
 such that
\begin{eqnarray}
\label{eq:underlineA-discounted}
g_{\underline{A}(B),B}(x_2(\underline{A}(B),B))=\ell.
\end{eqnarray}
Furthermore,
for $B\in(-\infty, \overline{A}\wedge \ell \beta)$,
\begin{eqnarray}
\label{eq:dunderlineA/dB-discounted}
\frac{d\underline{A}(B)}{d
  B}&=&-\frac{\lambda_1}{\lambda_2}e^{-(\lambda_1+\lambda_2)(x_2(\overline{A}(B),B)-a)}<0. \end{eqnarray}
Therefore, function $A=\underline{A}(B)$ is strictly decreasing in
$B\in(-\infty, \overline{A}\wedge \ell \beta)$; see Figure
\ref{fig:dis1b} for an illustration.
For $A\in(\underline{A}(B),\overline{A})$,
\begin{eqnarray}
\label{eq:g(x2)>l-discounted}
 g_{A,B}(x_2(A,B))>\ell.
\end{eqnarray}

(b) For each
\begin{equation}
  \label{eq:Arange}
  A \in \bigl(\underline{B}\vee(-k\beta), +\infty\bigr),
\end{equation}
there exists a unique
\begin{displaymath}
 \overline{B}(A) \in \bigl(\underline{B}, A\wedge h'(a+)\bigr)
\end{displaymath}
such that
\begin{eqnarray}
\label{eq:overlineB-discounted}
g_{A,\overline{B}(A)}(x_1(A,\overline{B}(A)))=-k.
\end{eqnarray}
Furthermore,
for $A\in(\underline{B}\vee (-k \beta),  \infty)$,
\begin{eqnarray}
\frac{d\overline{B}(A)}{d
  A}&=&-\frac{\lambda_2}{\lambda_1}e^{(\lambda_1+\lambda_2)(x_1(A,\overline{B}(A))-a)}<0.\label{eq:doverlineB/dA-discounted}
\end{eqnarray}
Therefore,  function $B=\overline{B}(A)$ is strictly decreasing in
$A\in (\underline{B}\vee (-k \beta),  \infty)$; see Figure
\ref{fig:dis1a} for an illustration.
For $B\in(\underline{B},\overline{B}(A))$,
\begin{eqnarray}
\label{eq:g(x1)<-k-discounted}
 g_{A,B}(x_1(A,B))<-k.
\end{eqnarray}

\begin{figure}[tb]
  \centering
  \includegraphics[width=14cm]{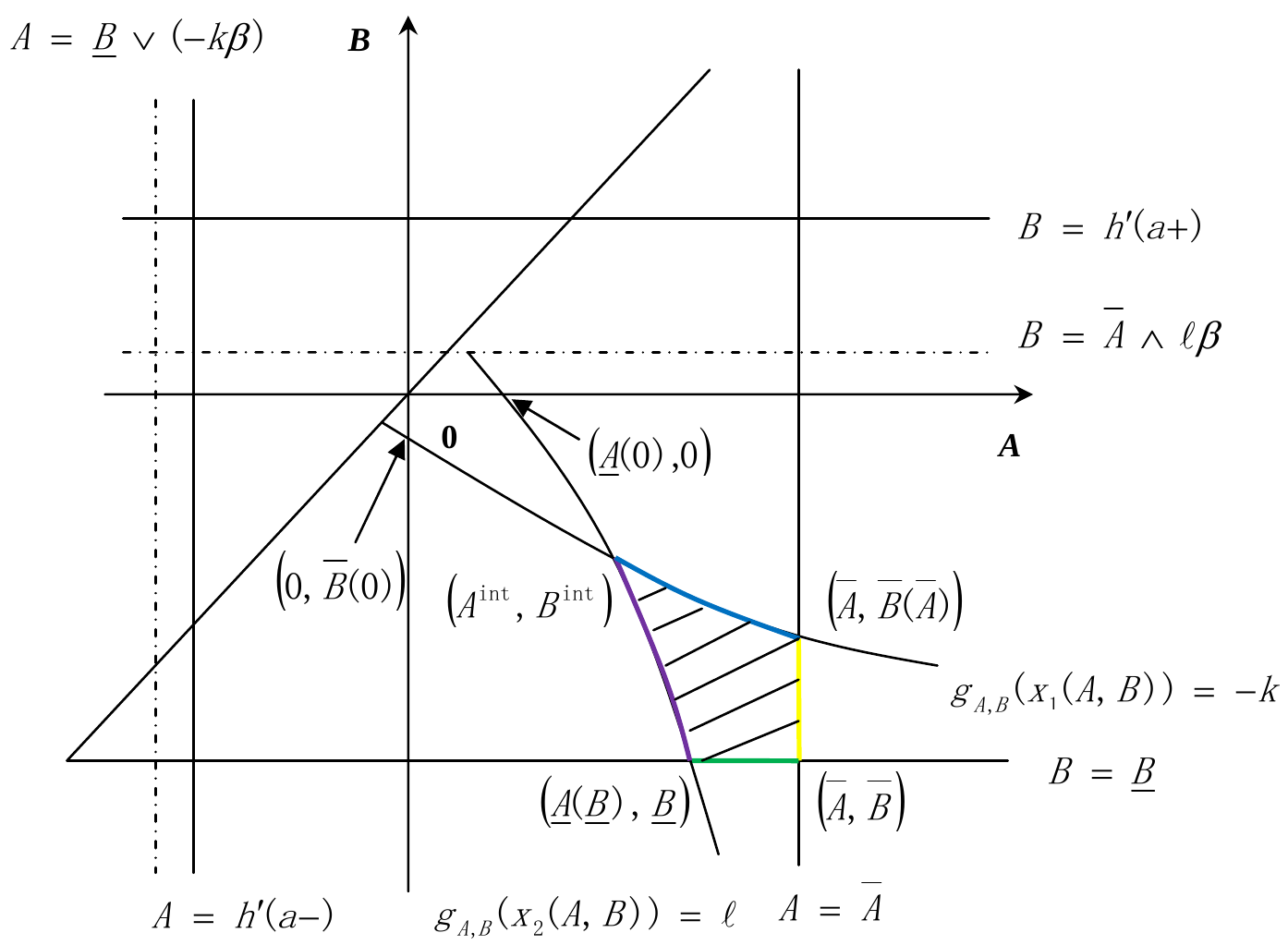}
\caption{The two curves $\{(\underline{A}(B),B):B\in[\underline{B},0]\}$ and
$\{(A,\overline{B}(A)):A\in[0,\overline{A}]\}$ have a unique intersection point
$(A^{\rm int},B^{\rm int})$ that satisfies $0<\underline{A}(0)<A^{\rm int}<\underline{A}(\underline{B})<\overline{A}$
and $0>\overline{B}(0)>B^{\rm int}>\overline{B}(\overline{A})>\underline{B}$.
For any $(A,B)$ in the shaded region $g_{A,\overline{B}(A)}(x_1(A,\overline{B}(A)))<-k$ and
$g_{\underline{A}(B),B}(x_2(\underline{A}(B),B))>\ell$.}
\label{fig:dis3}
\end{figure}

(c)
The two curves
 $\{(\underline{A}(B),B): B\in (-\infty, 0]\}$ and
 $\{(A,\overline{B}(A)): A\in [0, \infty)\}$ have a unique
 intersection point $(A^{\rm int}, B^{\rm int})$ that satisfies
 \begin{equation}
   \label{eq:ABintfix}
   \overline{B}(A^{\rm int}) = B^{\rm int} \quad \text{ and }
   \underline{A}(B^{\rm int}) = A^{\rm int}
 \end{equation}
with
\begin{eqnarray}
&& 0<\underline{A}(0)< A^{\rm int} <\underline{A}(\underline{B}) <
\overline{A},\label{eq:Aint}\\
&&
0>\overline{B}(0)> B^{\rm int} >\overline{B}(\overline{A}) >
\underline{B}.\label{eq:Bint}
\end{eqnarray}
See Figure \ref{fig:dis3} for an illustration.
\end{lemma}
\begin{proof}
(a) First, fix a $B$ that satisfies (\ref{eq:Brange}).
We consider the value of
$g_{A,B}(x_2(A,B))$ for $A\in (B\vee h'(a-), \overline{A})$.
\begin{eqnarray}
\frac{\partial g_{A,B}(x_2(A,B))}{\partial A}&=&g_{A,B}'(x_2(A,B))\frac{\partial x_2(A,B)}{\partial A}
+\frac{2}{\sigma^2}\frac{1}{\lambda_1+\lambda_2}\frac{1}{\lambda_1}e^{\lambda_1 (x_2(A,B)-a)}\nonumber\\
&=&\frac{2}{\sigma^2}\frac{1}{\lambda_1+\lambda_2}\frac{1}{\lambda_1}e^{\lambda_1 (x_2(A,B)-a)}\nonumber\\
&>&0.\label{eq:dg/dA-discounted}
\end{eqnarray}
Next we will prove that
\begin{eqnarray}
\label{eq:limg(x2)>l-discounted}
\lim_{A\uparrow \overline{A}}  g_{A,B}(x_2(A,B))>\ell
\end{eqnarray}
and
\begin{eqnarray}
\label{eq:limg(x2)<l-discounted}
\lim_{A\downarrow (B\vee h'(a-))}  g_{A,B}(x_2(A,B))<\ell,
\end{eqnarray}
from which one has that
there exists unique
$\underline{A}(B)\in (B\vee h'(a-),\overline{A})$ such that
\begin{eqnarray*}
 g_{\underline{A}(B),B}(x_2(\underline{A}(B),B))=\ell
\end{eqnarray*}
and for $A\in(\underline{A}(B),\overline{A})$
\begin{eqnarray*}
 g_{A,B}(x_2(A,B))>\ell.
\end{eqnarray*}
The derivative \eqref{eq:dunderlineA/dB-discounted} follows from the Implicit
Function Theorem, being applied  to \eqref{eq:underlineA-discounted}.

First we prove \eqref{eq:limg(x2)>l-discounted}.
Then \eqref{eq:x2ABlimit2-discounted} implies that
\begin{eqnarray}
&&\lim_{A\uparrow   \overline{A}}\bigl(B+\lambda_2\int_a^{x_2(A,B)} e^{\lambda_2 (y-a)}h'(y)dy\bigr) e^{-\lambda_2 (x_2(A,B)-a)}\nonumber\\
&&\quad=\lim_{A\uparrow \overline{A}}\lambda_2\int_a^{x_2(A,B)} e^{\lambda_2 (y-a)}h'(y)dy\cdot e^{-\lambda_2 (x_2(A,B)-a)}\nonumber\\
&&\quad=\lim_{x\uparrow \infty}\frac{\lambda_2\int_a^{x} e^{\lambda_2 (y-a)}h'(y)dy}{ e^{\lambda_2 (x-a)}}\nonumber\\
&&\quad=\lim_{x\uparrow \infty} h'(x),\label{eq:lim=h'-discounted}
\end{eqnarray}
where the last equlity follows from (\ref{eq:lim=h'-discounted2}).
Equalities \eqref{eq:first=second-discounted} and \eqref{eq:lim=h'-discounted} yield that
\begin{eqnarray*}
\lim_{A\uparrow   \overline{A}}\bigl(A-\lambda_1\int_a^{x_2(A,B)} e^{-\lambda_1 (y-a)}h'(y)dy\bigr)e^{\lambda_1 (x_2(A,B)-a)}
=\lim_{x\uparrow \infty} h'(x)
\end{eqnarray*}
Therefore, using the expression in \eqref{eq:g2-discounted} for $g$, we have
\begin{eqnarray*}
&&\lim_{A\uparrow \overline{A}}  g_{A,B}(x_2(A,B))\nonumber\\
&&\quad=
\frac{2}{\sigma^2}\frac{1}{\lambda_1+\lambda_2}\bigl[(\frac{1}{\lambda_1}+\frac{1}{\lambda_2})
\lim_{x\uparrow \infty}h'(x)\bigr]\nonumber\\
&&\quad =\frac{1}{\beta}\lim_{x\uparrow\infty}h'(x)\nonumber\\
&&\quad>\ell,
\end{eqnarray*}
where the second equality uses $\lambda_1\lambda_2=\frac{2\beta}{\sigma^2}$
and the last inequality is due to the first part of \eqref{eq:hlimit}.

It remains to prove \eqref{eq:limg(x2)<l-discounted}.
Next we consider two cases: $B\in (h'(a-),\overline{A}\wedge \ell
\beta)$ and $B\in (-\infty, h'(a-)]$.
If $B\in (h'(a-),\overline{A}\wedge \ell \beta)$,
$\lim_{A\downarrow B}x_2(A,B)=a$ in \eqref{eq:x2ABlimit1-discounted}
implies that
\begin{eqnarray}
\lim_{A\downarrow B}  g_{A_2,B_2}(x_2(A,B))&=&
\frac{2}{\sigma^2}\frac{1}{\lambda_1+\lambda_2}\bigl[(\frac{1}{\lambda_1}+\frac{1}{\lambda_2})(B-h'(a+))\bigr]+\frac{1}{\beta}h'(a+)\nonumber\\
&=&\frac{B}{\beta}.\label{eq:lim g(x1)-discounted}
\end{eqnarray}
Because $B<\ell \beta$, we have $\lim_{A\downarrow B}
g_{A,B}(x_2(A,B))<\ell$.

On the other hand, if $B\in (-\infty, h'(a-)]$,  \eqref{eq:first=second-discounted} implies that
\begin{eqnarray*}
&&\lim_{A\downarrow h'(a-)}\bigl(B+\lambda_2\int_a^{x_2(A,B)} e^{\lambda_2 (y-a)}h'(y)dy\bigr)e^{-\lambda_2 (x_2(A,B)-a)}\\
&&\quad=\lim_{A\downarrow h'(a-)}\bigl(A-\lambda_1\int_a^{x_2(A,B)} e^{-\lambda_1 (y-a)}h'(y)dy\bigr)e^{\lambda_1 (x_2(A,B)-a)}\\
&&\quad \leq 0,
\end{eqnarray*}
where the inequality is because $h'(a-)\leq 0$ by \eqref{eq:h'(a-)<h'(a+)-discounted}.
Using the expression in \eqref{eq:g2-discounted} for $g$, we have
\begin{eqnarray*}
\lim_{A\downarrow h'(a-)}g_{A,B}(x_2(A,B))\leq 0<\ell.
\end{eqnarray*}

(b) For a fixed $A$ that satisfies (\ref{eq:Arange}).
We can prove similarly that for $B\in (\underline{B}, A\wedge h'(a+))$
\begin{eqnarray}
&&\frac{\partial g_{A,B}(x_1(A,B))}{\partial B}=\frac{2}{\sigma^2}\frac{1}{\lambda_1+\lambda_2}\frac{1}{\lambda_2}
e^{-\lambda_2 (x_1(A,B)-a)}>0\label{eq:dg/dA2-discounted}
\end{eqnarray}
and
\begin{eqnarray}
\label{eq:limg(x1)-discounted}
\lim_{B\downarrow \underline{B}} g_{A,B}(x_1(A,B))
=\frac{1}{\beta}\lim_{x\downarrow -\infty}h'(x)<-k,
\end{eqnarray}
where the inequality is due to the second part of \eqref{eq:hlimit}.

If $A\in(\underline{B}\vee(-k\beta), h'(a+))$, we have
\begin{eqnarray}
\lim_{B\uparrow A}  g_{A,B}(x_1(A,B))=\frac{A}{\beta}.
\end{eqnarray}
Because $A>-k\beta$, we have $\lim_{B\uparrow A}  g_{A,B}(x_1(A,B))>-k$.
Then \eqref{eq:dg/dA2-discounted} and \eqref{eq:limg(x1)-discounted} imply that there exists a unique
$\overline{B}(A)\in (\underline{B},A)$ such that
\begin{eqnarray*}
 g_{A,\overline{B}(A)}(x_1(A,\overline{B}(A)))=-k
\end{eqnarray*}
and for $B\in(\underline{B},\overline{B}(A))$
\begin{eqnarray*}
 g_{A,B}(x_1(A,B))<-k.
\end{eqnarray*}

If $A\in [h'(a+),\infty)$, we have
\begin{eqnarray*}
\lim_{B\uparrow h'(a+)}g_{A,B}(x_1(A,B))\geq 0>-k.
\end{eqnarray*}
Then \eqref{eq:dg/dA2-discounted} implies that there exists a unique
$\overline{B}(A)\in (\underline{B},h'(a+))$ such that
\begin{eqnarray*}
 g_{A,\overline{B}(A)}(x_1(A,\overline{B}(A)))=-k
\end{eqnarray*}
and for $B\in(\underline{B},\overline{B}(A))$
\begin{eqnarray*}
 g_{A,B}(x_1(A,B))<-k.
\end{eqnarray*}
Applying the Implicit Function Theorem to
\eqref{eq:overlineB-discounted},
we also have \eqref{eq:doverlineB/dA-discounted}.

(c)
First consider the curve $\{(\underline{A}(B),B): B\in (-\infty,
\overline{A}\wedge \ell \beta)\}$ that is determined by
equation $g_{\underline{A}(B),B}(x_2(\underline{A}(B),B))=\ell$.
Consider two points
\begin{displaymath}
  (\underline{A}(0),0) \quad \text{ and } \quad
(\underline{A}(\underline{B}),\underline{B})
\end{displaymath}
on the curve  $\{(\underline{A}(B),B): B\in
(-\infty, \overline{A}\wedge \ell \beta)\}$ (see Figure~\ref{fig:dis3}).
By part (a) of this lemma, we have
\begin{eqnarray}
\label{eq:underlineA<overlineA-discounted}
  \underline{A}(\underline{B})<\overline{A}.
\end{eqnarray}
Next we show that
\begin{eqnarray}
\label{eq:underlineA(0)-discounted}
\underline{A}(0)>0.
\end{eqnarray}
To see this,
\eqref{eq:lim g(x1)-discounted} implies that
\begin{eqnarray*}
\lim_{A\downarrow 0} g_{A,0}(x_2(A,0))=0<\ell,
\end{eqnarray*}
from which and \eqref{eq:dg/dA-discounted}, one has  (\ref{eq:underlineA(0)-discounted}).

Similarly, consider two points
\begin{displaymath}
   (0,\overline{B}(0)) \quad \text{and}\quad
   (\overline{A},\overline{B}(\overline{A}))
\end{displaymath}
 on the curve determined by
 $g_{A, \overline{B}(A)}(x_1(A,\overline{B}(A)))=-k$.
 Similar to \eqref{eq:underlineA<overlineA-discounted} and
 \eqref{eq:underlineA(0)-discounted}, by part (b) of this lemma, we have
\begin{eqnarray}
\label{eq:overlineB(0)-discounted}
\underline{B}  < \overline{B}(\overline{A})<\overline{B}(0)<0.
\end{eqnarray}
Therefore, the point $   (\overline{A},\overline{B}(\overline{A}))$ is
on the right side of the curve
$g_{\underline{A}(B), B}(x_2(\underline{A}(B),B)=\ell$
and point $(0,\overline{B}(0))$ is on the left side of the curve. The
continuity and monotonicity of the two curves imply that there is a
unique point
\begin{displaymath}
  (A^{\rm int}, B^{\rm int})
\end{displaymath}
at which the two curves intersect. See  Figure~\ref{fig:dis3} for an
illustration. It is clear from Figure~\ref{fig:dis3} that
(\ref{eq:Aint}) and (\ref{eq:Bint}) hold.

\end{proof}

Let
\begin{equation}
  \label{eq:G}
  G =\{(A, B): \underline{A}(B)< A< \overline{A}, \quad \underline{B}< B <
  \overline{B}(A)\}
\end{equation}
be the shaded region in Figure \ref{fig:dis3}. The region $G$ has four
corners. They are $(A^{\rm int}, B^{\rm int})$, $(\overline{A},
\overline{B}(\overline{A}))$, $(\overline{A}, \underline{B})$ and
$(\underline{A}(\underline{B}), \underline{B})$. Its boundary has four
pieces: the top, the right, the bottom and the left.

%  For $(A, B)\in G$,
% there exist unique $d(A,B)$, $D(A,B)$, $U(A,B)$ and $u(A,B)$ such that
% \begin{eqnarray*}
% &&d(A,B)<x_1(A,B)<D(A,B)<U(A,B)<x_2(A,B)<u(A,B),\\
% && g_{A,B}(d(A,B))=g_{A,B}(D(A,B))=-k,\\
% && g_{A,B}(U(A,B))=g_{A,B}(u(A,B))=\ell,\\
% && g_{A,B}'(d(A,B))<0,\quad g_{A,B}'(D(A,B))>0,\\
% &&g_{A,B}'(U(A,B))>0,\quad g_{A,B}'(u(A,B))<0.
% \end{eqnarray*}
% \end{lemma}
% \begin{proof}

For $(A, B)\in G$,
we have
\begin{equation}
  \label{eq:gx1x2}
g_{A,B}(x_1(A,B))<-k, \quad g_{A,B}(x_2(A,B))>\ell.
\end{equation}
It follows from part (b) of  Lemma \ref{lem:optimalParameter-discounted}
and (\ref{eq:gx1x2})  that
there exist unique $d(A,B)$, $D(A,B)$, $U(A,B)$ and $u(A,B)$ such that
\begin{eqnarray*}
&&d(A,B)<x_1(A,B)<D(A,B)<U(A,B)<x_2(A,B)<u(A,B),\\
&& g_{A,B}(d(A,B))=g_{A,B}(D(A,B))=-k,\\
&& g_{A,B}(U(A,B))=g_{A,B}(u(A,B))=\ell,\\
&& g_{A,B}'(d(A,B))<0,\quad g_{A,B}'(D(A,B))>0,\\
&&g_{A,B}'(U(A,B))>0,\quad g_{A,B}'(u(A,B))<0.
\end{eqnarray*}
For each $(A, B)\in G$, define
\begin{displaymath}
\Lambda_1(A,B)=\int_{d(A,B)}^{D(A,B)} \bigl[g_{A,B}(x)+k\bigr]dx,
\quad \text{and} \quad
\Lambda_2(A,B)=\int_{U(A,B)}^{u(A,B)} \bigl[g_{A,B}(x)-\ell\bigr]dx.
\end{displaymath}
Although $(A, \overline{B}(A))$ is not in $G$ for $A\in (A^{\rm int},
\overline{A})$, these points are on the upper boundary of $G$, and
\begin{displaymath}
  \Lambda_2(A, \overline{B}(A))
\end{displaymath}
is also well defined for
 $A\in (A^{\rm int},\overline{A})$.

\begin{figure}[tb]
  \centering
  \includegraphics[width=14cm]{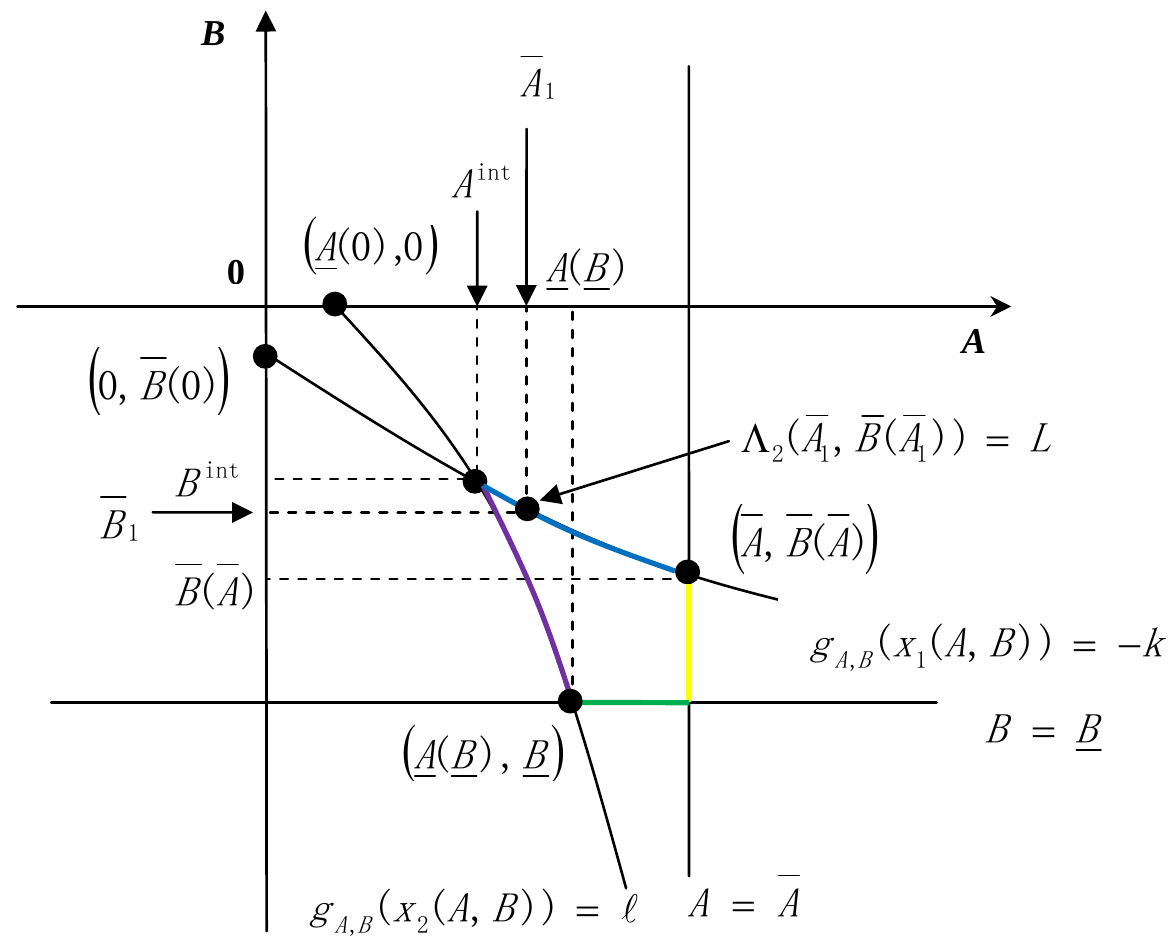}
\caption{The point $(\overline{A}_1,\overline{B}_1)$ is the unique point on the
top boundary of $G$ such that $\Lambda_2(\overline{A}_1,\overline{B}_1)=L$.
For any point $(A,B)$ on the top boundary and to the right of $(\overline{A}_1,\overline{B}_1)$,
$\Lambda_2(A,B)>L$.}
\label{fig:dis4}
\end{figure}

\begin{lemma}
\label{lem:A^*(B)-discounted}
There exists a unique $A=\overline{A}_1\in (A^{\rm int}, \overline{A})$
such that
\begin{eqnarray}
\label{eq:Lambda2=L-discounted}
\Lambda_2(\overline{A}_1, \overline{B}(\overline{A}_1))=L,
\end{eqnarray}
and for $A\in(\overline{A}_1,\overline{A})$,
\begin{eqnarray}
\label{eq:Lambda2>L-discounted}
\Lambda_2(A, \overline{B}(A))>L.
\end{eqnarray}
\end{lemma}

\begin{proof}

When $A$ goes to $A^{\rm int}$, $(A,\overline{B}(A))$ goes to $(A^{\rm int},B^{\rm int})$.
Then the definition of $(A^{\rm int},B^{\rm int})$ in Lemma \ref{lem:underlineAoverlineB-discounted} implies that
\begin{eqnarray*}
\lim_{A\downarrow A^{\rm int}}U(A,\overline{B}(A))=\lim_{A\downarrow A^{\rm int}}u(A,\overline{B}(A))=x_2(A^{\rm int},B^{\rm int}).
\end{eqnarray*}
Therefore,
\begin{equation}
  \label{eq:Aintlimit0}
\lim_{A\downarrow A^{\rm int}}\Lambda_2(A,\overline{B}(A))=0.
\end{equation}

Fix $A\in(A^{\rm int},\overline{A})$. One has
\begin{eqnarray}
  &&\frac{\partial \Lambda_2(A,\overline{B}(A))}{\partial A}\nonumber\\
  &&\quad=\frac{\partial u(A,\overline{B}(A))}{\partial A}\bigl[g_{A,\overline{B}(A)}(u(A,\overline{B}(A)))-\ell\bigr]\nonumber\\
  &&\quad\quad-\frac{\partial U(A,\overline{B}(A))}{\partial A}\bigl[g_{A,\overline{B}(A)}(U(A,\overline{B}(A)))-\ell\bigr]
  +\int_{U(A,\overline{B}(A))}^{u(A,\overline{B}(A))}\frac{\partial g_{A,\overline{B}(A)}(x)}{\partial A}dx\nonumber\\
  &&\quad=\int_{U(A,\overline{B}(A))}^{u(A,\overline{B}(A))}\frac{\partial g_{A,\overline{B}(A)}(x)}{\partial A}dx\nonumber\\
  &&\quad=\frac{2}{\sigma^2}\frac{1}{\lambda_1+\lambda_2}\int_{U(A,\overline{B}(A))}^{u(A,\overline{B}(A))}\bigl[\frac{1}{\lambda_1}e^{\lambda_1 (x-a)}+\frac{1}{\lambda_2}\frac{d \overline{B}(A)}{d A}e^{-\lambda_2 (x-a)}\bigr]dx\nonumber\\
  &&\quad=\frac{2}{\sigma^2}\frac{1}{\lambda_1+\lambda_2}\int_{U(A,\overline{B}(A))}^{u(A,\overline{B}(A))}\bigl[\frac{1}{\lambda_1}e^{\lambda_1 (x-a)} -\frac{1}{\lambda_1}e^{(\lambda_1+\lambda_2)(x_1(A,\overline{B}(A))-a)}e^{-\lambda_2 (x-a)}\bigr]dx\nonumber\\
  &&\quad >0,\label{eq:dLambda2/dA-discounted}
\end{eqnarray}
where the second equality is due to $g_{A,\overline{B}(A)}(U(A,\overline{B}(A)))=g_{A,\overline{B}(A)}(u(A,\overline{B}(A)))=\ell$,
the forth equality is from \eqref{eq:doverlineB/dA-discounted},
and the inequality is due to
$u(A,\overline{B}(A))>U(A,\overline{B}(A))>x_1(A,\overline{B}(A))$. Therefore
$\Lambda_2(A,\overline{B}(A))$ is increasing in $A\in (A^{\rm int},
\overline{A})$.

We will show next that
\begin{eqnarray}
\label{eq:limLambda2-discounted}
\lim_{A\uparrow \overline{A}}\Lambda_2(A,\overline{B}(A))=\infty.
\end{eqnarray}
It follows from (\ref{eq:Aintlimit0}),
(\ref{eq:limLambda2-discounted}) and the monotonicity of
$\Lambda_2(A,\overline{B}(A))$ that
there exists unique $\overline{A}_1 \in(A^{\rm int},\overline{A})$
such that
(\ref{eq:Lambda2=L-discounted}) and \eqref{eq:Lambda2>L-discounted} hold.

To prove (\ref{eq:limLambda2-discounted}), note that \eqref{eq:doverlineB/dA-discounted} implies that
\begin{eqnarray}
&&\frac{\partial g_{A,\overline{B}(A)}(x_2(A,\overline{B}(A)))}{\partial A}\nonumber\\
&&\quad=g_{A,\overline{B}(A)}'(x_2(A,\overline{B}(A)))\frac{\partial x_2(A,\overline{B}(A))}{\partial A}\nonumber\\
&&\quad\quad+\frac{2}{\sigma^2}\frac{1}{\lambda_1+\lambda_2}\bigl[\frac{1}{\lambda_1}e^{\lambda_1 (x_2(A,\overline{B}(A))-a)}
+\frac{1}{\lambda_2}\frac{d \overline{B}(A)}{d A}e^{-\lambda_2 (x_2(A,\overline{B}(A))-a)}\bigr]\nonumber\\
&&\quad=\frac{2}{\sigma^2}\frac{1}{\lambda_1+\lambda_2}\bigl[\frac{1}{\lambda_1}e^{\lambda_1 (x_2(A,\overline{B}(A))-a)} -\frac{1}{\lambda_1}e^{(\lambda_1+\lambda_2)(x_1(A,\overline{B}(A))-a)}e^{-\lambda_2 (x_2(A,\overline{B}(A))-a)}\bigr]\nonumber\\
&&\quad>0,\label{eq:dg(x2)/dA-discounted}
\end{eqnarray}
where the second equality is due to $g_{A,\overline{B}(A)}'(x_2(A,\overline{B}(A)))=0$,
and the inequality is due to $x_2(A,\overline{B}(A))>x_1(A,\overline{B}(A))$.

For $A\in(A^{\rm int},\overline{A})$,
$(A, \overline{B}(A))$ on the right side of the curve
$g_{A,B}(x_2(A,B))=\ell$ and therefore
\begin{eqnarray}
\label{eq:g>l-disocunted}
g_{A,\overline{B}(A)}(x_2(A,\overline{B}(A)))>\ell.
\end{eqnarray}
Fix an $A'\in(A^{\rm int},\overline{A})$ and let
\begin{eqnarray*}
M_1=\Big(g_{A',\overline{B}(A')}(x_2(A',\overline{B}(A')))-\ell\Big)/2.
\end{eqnarray*}
It follows from \eqref{eq:g>l-disocunted} that $M_1>0$.
Then \eqref{eq:dg(x2)/dA-discounted} implies that for each $A\in(A',\overline{A})$,
\begin{eqnarray*}
g_{A,\overline{B}(A)}(x_2(A,\overline{B}(A)))
>g_{A',\overline{B}(A')}(x_2(A',\overline{B}(A')))
=\ell+2M_1>\ell+M_1.
\end{eqnarray*}
Therefore, for each $A\in(A',\overline{A})$,
there exist unique $U_1(A,\overline{B}(A))$ and $u_1(A,\overline{B}(A))$ such that
\begin{eqnarray}
&&U_1(A,\overline{B}(A))<x_2(A,\overline{B}(A))<u_1(A,\overline{B}(A))\nonumber\\
&& g_{A,\overline{B}(A)}(U_1(A,\overline{B}(A)))= g_{A,\overline{B}(A)}(u_1(A,\overline{B}(A)))=\ell+M_1,\label{eq:g=l+M-discounted}\\
&& g_{A,\overline{B}(A)}'(U_1(A,\overline{B}(A)))>0,\  g_{A,\overline{B}(A)}'(u_1(A,\overline{B}(A)))<0.\nonumber
\end{eqnarray}
The properties of $g_{A, B}$ in Lemma \ref{lem:optimalParameter-discounted} imply that
for $A\in(A',\overline{A})$,
\begin{eqnarray*}
U(A,\overline{B}(A))<U_1(A,\overline{B}(A))<x_2(A,\overline{B}(A))<u_1(A,\overline{B}(A))<u(A,\overline{B}(A)).
\end{eqnarray*}
This implies that
\begin{eqnarray}
\lim_{A\uparrow \overline{A}} u_1(A,\overline{B}(A))
&\geq&\lim_{A\uparrow \overline{A}} x_2(A,\overline{B}(A))\nonumber\\
&\geq&\lim_{A\uparrow \overline{A}} x_2(A,B^{\rm int})\nonumber\\
&=&\infty,\label{eq:limu1-discounted}
\end{eqnarray}
where the second inequality holds because \eqref{eq:dx2/dB-discounted} and the equality is due to \eqref{eq:x2ABlimit2-discounted}.
Therefore, for $A\in(A',\overline{A})$,
\begin{eqnarray*}
\Lambda_2(A,\overline{B}(A))
&=&\int_{U(A,\overline{B}(A))}^{u(A,\overline{B}(A))}\bigl[g_{A,\overline{B}(A)}(x)-\ell\bigr]dx\\
&\geq& \int_{U_1(A,\overline{B}(A))}^{u_1(A,\overline{B}(A))}\bigl[g_{A,\overline{B}(A)}(x)-\ell\bigr]dx\\
&\geq&M_1(u_1(A,\overline{B}(A))-U_1(A,\overline{B}(A))).
\end{eqnarray*}
Applying the Implicit Function Theorem to $g_{A,\overline{B}(A)}(U_1(A,\overline{B}(A)))=\ell+M_1$, we have
that
\begin{eqnarray*}
&&\frac{\partial U_1(A,\overline{B}(A))}{\partial A}\\
&&\quad=-\frac{\frac{2}{\sigma^2}\frac{1}{\lambda_1+\lambda_2}\Big[\frac{1}{\lambda_1}e^{\lambda_1 (U_1(A,\overline{B}(A))-a)}
+\frac{1}{\lambda_2}\frac{d \overline{B}(A)}{d A}e^{-\lambda_2 (U_1(A,\overline{B}(A))-a)}\Big]}{g_{A,\overline{B}(A)}'(U_1(A,\overline{B}(A)))}\\
&&\quad=-\frac{\frac{2}{\sigma^2}\frac{1}{\lambda_1+\lambda_2}\Big[\frac{1}{\lambda_1}e^{\lambda_1 (U_1(A,\overline{B}(A))-a)}
-\frac{1}{\lambda_1}e^{(\lambda_1+\lambda_2)(x_1(A,\overline{B}(A))-a)}e^{-\lambda_2 (U_1(A,\overline{B}(A))-a)}\Big]}{g_{A,\overline{B}(A)}'(U_1(A,\overline{B}(A)))}\\
&&\quad<0,
\end{eqnarray*}
where the second equality is due to \eqref{eq:doverlineB/dA-discounted}, and
the inequality is due to $U_1(A,\overline{B}(A))>x_1(A,\overline{B}(A))$ and
$g_{A,\overline{B}(A)}'(U_1(A,\overline{B}(A)))>0$.
Thus, for any  $A\in(A',\overline{A})$,
\begin{eqnarray*}
U_1(A,\overline{B}(A))\leq U_1(A',\overline{B}(A')).
\end{eqnarray*}
Therefore, for any  $A\in(A',\overline{A})$,
\begin{eqnarray*}
\Lambda_2(A,\overline{B}(A))\geq M_1(u_1(A,\overline{B}(A))-U_1(A',\overline{B}(A'))),
\end{eqnarray*}
which, together with \eqref{eq:limu1-discounted}, implies
(\ref{eq:limLambda2-discounted}).
\end{proof}

Define
\begin{eqnarray}
\label{eq:defoverlineB1-discounted}
\overline{B}_1=\overline{B}(\overline{A}_1).
\end{eqnarray}
It follows from
\eqref{eq:doverlineB/dA-discounted}  that
\begin{eqnarray*}
\underline{B} < \overline{B}(\overline{A}) <
\overline{B}_1 < B^{\rm int}<0.
\end{eqnarray*}
where the last inequality is due to \eqref{eq:Bint}. See
Figure~\ref{fig:dis4} for the point $(\overline{A}_1,\overline{B}_1)$.

\begin{lemma}
\label{lem:A^*(B)-1discounted}
(a) For $B\in(\underline{B},\overline{B}_1]$,
there exists unique $A^*(B)\in[\overline{A}_1,\overline{A})$ such that
\begin{eqnarray}
\label{eq:Lambda2*=L-discounted}
\Lambda_2(A^*(B), B)=L.
\end{eqnarray}

(b)  For $B\in(\underline{B},\overline{B}_1]$,
\begin{eqnarray}
\label{eq:dA^*(B)/dB-discounted}
\frac{dA^*(B)}{d B}
=\frac{\lambda_1^2(e^{-\lambda_2 (u(A^*(B),B)-a)}-e^{-\lambda_2
    (U(A^*(B),B)-a)})}{\lambda_2^2(e^{\lambda_1 (u(A^*(B),B)-a)}-e^{\lambda_1
    (U(A^*(B),B)-a)})}
 <0.
\end{eqnarray}
\end{lemma}

\begin{proof}
(a) For $B=\overline{B}_1$, Lemma \ref{lem:A^*(B)-discounted} has showed that
\begin{eqnarray}
\label{eq:A^*=A1-discounted}
A^*(\overline{B}_1)=\overline{A}_1.
\end{eqnarray}
For $B\in(\underline{B},\overline{B}_1)$ and $(A, B)\in G$, we first have
\begin{eqnarray}
\frac{\partial \Lambda_2(A,B)}{\partial A}
&=&\int_{U(A,B)}^{u(A,B)}\frac{\partial g_{A,B}(x)}{\partial A}dx\nonumber\\
&=&\int_{U(A,B)}^{u(A,B)}\frac{2}{\sigma^2}\frac{1}{\lambda_1+\lambda_2}\frac{1}{\lambda_1}e^{\lambda_1 (x-a)}dx\nonumber\\
&>&0.\label{eq:dLambda2/dA-discounted1}
\end{eqnarray}
From the definition of $\underline{A}(B)$ in \eqref{eq:underlineA-discounted},
we have
\begin{eqnarray*}
\lim_{A\downarrow \underline{A}(B)} U(A,B)=\lim_{A\downarrow \underline{A}(B)} u(A,B)=
\lim_{A\downarrow \underline{A}(B)} x_2(A,B)=x_2(\underline{A}(B),B).
\end{eqnarray*}
Therefore, for a fixed $B\in(\underline{B},\overline{B}_1)$,
\begin{eqnarray}
\label{eq:Lambda2lessL}
\lim_{A\downarrow \underline{A}(B)}\Lambda_2(A,B)=0<L.
\end{eqnarray}
Next for $B\in(\underline{B},\overline{B}_1)$, we consider two cases depending on whether
$B\in (\underline{B},\overline{B}(\overline{A})]$ or $B\in ( \overline{B}(\overline{A}), \overline{B}_1)$.
See Figure \ref{fig:dis5} for an illustration.

We first assume that $B\in ( \overline{B}(\overline{A}), \overline{B}_1)$.
For a fixed $B\in ( \overline{B}(\overline{A}), \overline{B}_1)$,
by the monotonicity of $\overline{B}(\cdot)$ in \eqref{eq:doverlineB/dA-discounted},
there exists an $A(B)\in (\overline{A}_1, \overline{A})$ such that
$(A(B), B)$ is on the upper boundary of $G$.
It follows from  \eqref{eq:dLambda2/dA-discounted} and the definition of $\overline{A}_1$ in Lemma
\ref{lem:A^*(B)-discounted} that
\begin{equation*}
 \Lambda_2(A(B), B)=\Lambda_2(A(B), \overline{B}(A(B)))>\Lambda_2(\overline{A}_1, \overline{B}(\overline{A}_1))= L,
\end{equation*}
which, together with \eqref{eq:dLambda2/dA-discounted1} and \eqref{eq:Lambda2lessL},
implies that there exists a unique
\begin{displaymath}
A^*(B)\in(\underline{A}(B),A(B))
\end{displaymath}
 such that
\eqref{eq:Lambda2*=L-discounted} holds.

Now assume that $B\in (\underline{B},\overline{B}(\overline{A})]$.
Following the proof for \eqref{eq:limLambda2-discounted}, one can
prove similarly that
\begin{eqnarray*}
\lim_{A\uparrow \overline{A}}\Lambda_2(A,B)=\infty,
\end{eqnarray*}
which, together with \eqref{eq:dLambda2/dA-discounted1} and
\eqref{eq:Lambda2lessL}, implies that
there exists a unique $A^{*}(B)\in (\underline{A}(B), \overline{A})$ such that
\eqref{eq:Lambda2=L-discounted} holds.
By (\ref{eq:dunderlineA/dB-discounted}) and
\eqref{eq:A^*=A1-discounted}, we have
for $B\in(\underline{B},\overline{B}_1)$,
\begin{eqnarray*}
A^*(B)>\overline{A}_1.
\end{eqnarray*}

(b) Applying the Implicit Function Theorem to $\Lambda_2(A^*(B), B)=L$,
we have  \eqref{eq:dA^*(B)/dB-discounted}.

\end{proof}

\begin{figure}[t]
  \centering
  \includegraphics[width=14cm]{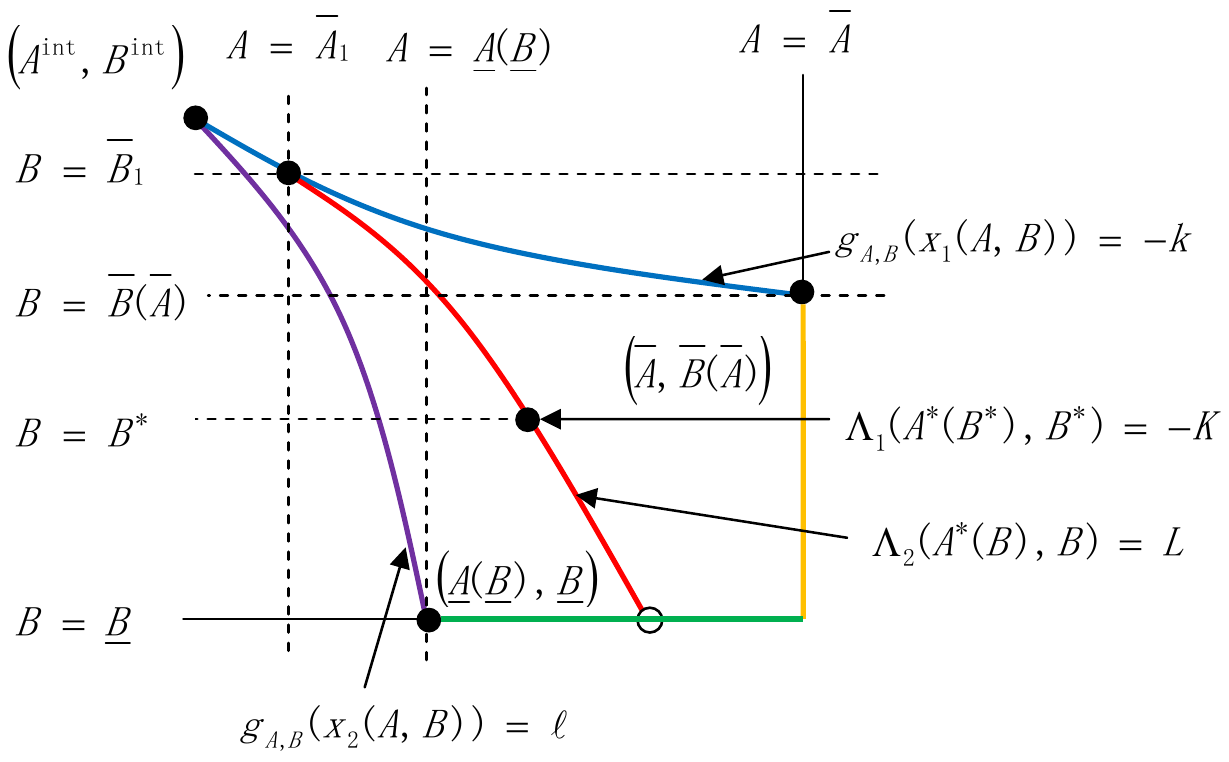}
\caption{For $B\in(\underline{B},\overline{B}_1]$,
there exists a unique $A^*(B)\in[\overline{A}_1,\overline{A})$ such that
$\Lambda_2(A^*(B), B)=L$.
There is a unique $B^*\in(\underline{B},\overline{B}_1)$
that satisfies $\Lambda_1(A^*(B^*),B^*)=-K$.}
\label{fig:dis5}
\end{figure}

For each $B\in (\underline{B},\overline{B}_1)$, Lemma
\ref{lem:A^*(B)-1discounted} shows that $(A^*(B), B)\in G$. Thus,
\begin{displaymath}
  g_{A^*(B), B}(x_1(A^*(B), B))< -k
\end{displaymath}
and
\begin{displaymath}
  \Lambda_1(A^*(B),B)  = \int_{d(A^*(B), B)}^{D(A^*(B), B)} [
  g_{A^*(B), B}(x) +k ] dx
\end{displaymath}
is well defined.

\begin{lemma}
\label{lem:B^*-discounted}
There exists a unique $B^*$ with $B^*\in(\underline{B},\overline{B}_1)$ such that
$\Lambda_1(A^*(B^*),B^*)=-K$.
\end{lemma}
\begin{proof}
We only need to show that $\Lambda_1(A^*(B),B)$ can take any value in
$(-\infty,0)$ for
$B\in(\underline{B},\overline{B}_1)$ and is strictly increasing in $B$.

It has been shown in Lemma \ref{lem:A^*(B)-discounted}  that
that $A^*(\overline{B}_1)=\overline{A}_1$ and
$(\overline{A}_1, \overline{B}_1)$ is on the upper boundary of $G$
(the blue curve in Figure \ref{fig:dis5}).
Therefore
\begin{displaymath}
 g_{\overline{A}_1,\overline{B}_1}(x_1(\overline{A}_1,\overline{B}_1))=-k
\end{displaymath}
and
\begin{eqnarray}
\lim_{B\uparrow \overline{B}_1}g_{A^*(B),B}(x_1(A^*(B),B))
=g_{A^*(\overline{B}_1),\overline{B}_1}(x_1(A^*(\overline{B}_1),\overline{B}_1))
=g_{\overline{A}_1,\overline{B}_1}(x_1(\overline{A}_1,\overline{B}_1))
=-k.\nonumber\\\label{eq:g=-k-discounted}
\end{eqnarray}
It follows that
\begin{displaymath}
\lim_{B\uparrow \overline{B}_1}\Lambda_1(A^*(B),B)=0.
\end{displaymath}
We now prove
\begin{eqnarray}
\label{eq:limLambda1-discounted}
\lim_{B\downarrow \underline{B}}\Lambda_1(A^*(B),B)=-\infty.
\end{eqnarray}

First, we prove
\begin{eqnarray}
\frac{\partial g_{A^*(B),B}(x_1(A^*(B),B))}{\partial B}
>0 \label{eq:derigless0}.
\end{eqnarray}
To see this, for  $B\in (\underline{B}, \overline{B}_1)$,
\begin{eqnarray*}
&&\frac{\partial g_{A^*(B),B}(x_1(A^*(B),B))}{\partial B}\\
&&\quad=\frac{2}{\sigma^2}\frac{1}{\lambda_1+\lambda_2}
\bigr[\frac{1}{\lambda_1}\frac{dA^*(B)}{d B}e^{\lambda_1
  (x_1(A^*(B),B)-a)}+\frac{1}{\lambda_2}e^{-\lambda_2
  (x_1(A^*(B),B)-a)}\bigr]\\
&&\quad=
\frac{2}{\sigma^2}\frac{1}{\lambda_1+\lambda_2}
\bigr[\frac{1}{\lambda_1}\frac{\lambda_1^2(e^{-\lambda_2 (u(A^*(B),B)-a)}-e^{-\lambda_2
    (U(A^*(B),B)-a)})}{\lambda_2^2(e^{\lambda_1 (u(A^*(B),B)-a)}-e^{\lambda_1 (U(A^*(B),B)-a)})}e^{\lambda_1 (x_1(A^*(B),B)-a)}\\
   &&\quad\quad+\frac{1}{\lambda_2}e^{-\lambda_2
     (x_1(A^*(B),B)-a)}\bigr],
\end{eqnarray*}
where the second equality follows from
\eqref{eq:dA^*(B)/dB-discounted}.
Using the {\em Lagrange Mean Value Theorem}, there exist $y_1\in\bigl(U(A^*(B),B),u(A^*(B),B)\bigr)$
and $y_2\in\bigl(U(A^*(B),B),u(A^*(B),B)\bigr)$ such that
\begin{eqnarray}
&&e^{-\lambda_2 (u(A^*(B),B)-a)}-e^{-\lambda_2 (U(A^*(B),B)-a)}
=-\lambda_2e^{-\lambda_2 (y_1-a)}\bigl(u(A^*(B),B)-U(A^*(B),B)\bigr),\nonumber\\
\label{eq:y1-discounted}\\
&&e^{\lambda_1 (u(A^*(B),B)-a)}-e^{\lambda_1 (U(A^*(B),B)-a)}=\lambda_1e^{\lambda_1 (y_2-a)}\bigl(u(A^*(B),B)-U(A^*(B),B)\bigr).\label{eq:y2-discounted}
\end{eqnarray}
Therefore, for $B\in (\underline{B}, \overline{B}_1)$,
\begin{eqnarray*}
&&\frac{\partial g_{A^*(B),B}(x_1(A^*(B),B))}{\partial B}\nonumber\\
&&\quad= \frac{2}{\sigma^2}\frac{1}{\lambda_1+\lambda_2}
\bigr[-\frac{1}{\lambda_2}\frac{e^{-\lambda_2 (y_1-a)}}{e^{\lambda_1 (y_2-a)}}e^{\lambda_1 (x_1(A^*(B),B)-a)}
    +\frac{1}{\lambda_2}e^{-\lambda_2 (x_1(A^*(B),B)-a)}\bigr]\nonumber\\
&&\quad= \frac{2}{\sigma^2}\frac{1}{\lambda_1+\lambda_2}\frac{1}{\lambda_2}
e^{-\lambda_2 (y_1-a)}\bigr[-e^{\lambda_1 (x_1(A^*(B),B)-y_2)}
    +e^{-\lambda_2 (x_1(A^*(B),B)-y_1)}\bigr]\nonumber\\
&&\quad >0,
\end{eqnarray*}
where the inequality holds because $x_1(A^*(B),B)<
D(A^*(B),B)<U(A^*(B),B)<y_1$ and $x_1(A^*(B),B)<
D(A^*(B),B)<U(A^*(B),B)<y_2$.  Thus, we have proved (\ref{eq:derigless0}).

Fix an $\overline{B}_2\in(\underline{B},\overline{B}_1)$.
Define
\begin{eqnarray*}
M_2=-\frac{g_{A^*(\overline{B}_2),\overline{B}_2}(x_1(A^*(\overline{B}_2),\overline{B}_2))+k}{2}.
\end{eqnarray*}
It follows
from \eqref{eq:g=-k-discounted} and \eqref{eq:derigless0}  that
$g_{A^*(\overline{B}_2),\overline{B}_2}(x_1(A^*(\overline{B}_2),\overline{B}_2))<-k$
and thus $M_2>0$.
From (\ref{eq:derigless0}), one has that,  for
$B\in(\underline{B},\overline{B}_2)$,
\begin{eqnarray*}
g_{A^*(B),B}(x_1(A^*(B),B))
<g_{A^*(\overline{B}_2),\overline{B}_2}(x_1(A^*(\overline{B}_2),\overline{B}_2))
=-k-2M_2<-k-M_2.
\end{eqnarray*}
Therefore, for $B\in(\underline{B},\overline{B}_2)$,
there exist unique $d_1(A^*(B),B)$ and $D_1(A^*(B),B)$ such that
\begin{eqnarray*}
&&d_1(A^*(B),B)<x_1(A^*(B),B)<D_1(A^*(B),B),\\
&&g_{A^*(B),B}(d_1(A^*(B),B))=g_{A^*(B),B}(D_1(A^*(B),B))=-k-M_2,\\
&&g_{A^*(B),B}'(d_1(A^*(B),B))<0, \quad g_{A^*(B),B}'(D_1(A^*(B),B))>0.
\end{eqnarray*}
The properties of $g_{A^*(B), B}$ in Lemma
\ref{lem:optimalParameter-discounted} imply that
for $B\in(\underline{B},\overline{B}_2)$,
\begin{eqnarray*}
d(A^*(B),B)<d_1(A^*(B),B)<x_1(A^*(B),B)<D_1(A^*(B),B)<D(A^*(B),B).
\end{eqnarray*}
Therefore, for $B\in(\underline{B},\overline{B}_2)$,
\begin{eqnarray*}
\Lambda_1(A^*(B),B)&=&\int_{d(A^*(B),B)}^{D(A^*(B),B)}
\bigl[g_{A^*(B),B}(x)+k\bigr]dx\\
&\leq&\int_{d_1(A^*(B),B)}^{D_1(A^*(B),B)} \bigl[g_{A^*(B),B}(x)+k\bigr]dx\\
&\leq& -M_2 (D_1(A^*(B),B)-d_1(A^*(B),B)).
\end{eqnarray*}
By \eqref{eq:x1ABlimit1-discounted}, \eqref{eq:dx1/dA-discounted} and
$A^*(B)\geq A^{\rm int}$, one has that
\begin{eqnarray}
\label{eq:limx1-discounted}
\lim_{B\downarrow \underline{B}}x_1(A^*(B),B)\leq \lim_{B\downarrow
  \underline{B}}x_1(A^{\rm int},B)=-\infty.
\end{eqnarray}
Because $d_1(A^*(B),B)<x_1(A^*(B),B)$,
(\ref{eq:limx1-discounted}) implies that
\begin{eqnarray}
\label{eq:limd2-discounted}
\lim_{B\downarrow \underline{B}}d_1(A^*(B),B)=-\infty.
\end{eqnarray}
Now we  prove
\begin{equation}
  \label{eq:D1limitfinite}
\lim_{B\downarrow \underline{B}}D_1(A^*(B),B)>-\infty,
\end{equation}
which, together with \eqref{eq:limd2-discounted}, implies that
\begin{eqnarray*}
&&\lim_{B\downarrow \underline{B}}\Lambda_1(A^*(B),B)\\
&&\quad\leq \lim_{B\downarrow \underline{B}} -M_2 (D_1(A^*(B),B)-d_1(A^*(B),B))\\
&&\quad=-\infty,
\end{eqnarray*}
proving (\ref{eq:limLambda1-discounted}).

To prove (\ref{eq:D1limitfinite}), noting the definitions of $y_1$ and $y_2$, we have that
for $B\in (\underline{B}, \overline{B}_1)$,
\begin{eqnarray*}
&&\frac{\partial D_1(A^*(B),B)}{\partial B}\\
&&\quad=\frac{-\frac{2}{\sigma^2}\frac{1}{\lambda_1+\lambda_2}
\Bigl[\frac{1}{\lambda_1}\frac{dA^*(B)}{d B}e^{\lambda_1 (D_1(A^*(B),B)-a)}+\frac{1}{\lambda_2}e^{-\lambda_2 (D_1(A^*(B),B)-a)}\Bigr]}{g_{A^*(B),B}'(D_1(A^*(B),B))}\\
&&\quad =-\frac{2}{\sigma^2}\frac{1}{\lambda_1+\lambda_2}
\frac{1}{g_{A^*(B),B}'(D_1(A^*(B),B))}
\biggl[
    \frac{1}{\lambda_2}e^{-\lambda_2 (D_1(A^*(B),B)-a)} + \\
&&\quad \quad
\frac{1}{\lambda_1}\frac{\lambda_1^2(e^{-\lambda_2 (u(A^*(B),B)-a)}-e^{-\lambda_2
    (U(A^*(B),B)-a)})}{\lambda_2^2(e^{\lambda_1
    (u(A^*(B),B)-a)}-e^{\lambda_1 (U(A^*(B),B)-a)})}
    e^{\lambda_1  (D_1(A^*(B),B)-a)}
\biggr]\\
&&\quad =-\frac{2}{\sigma^2}\frac{1}{\lambda_1+\lambda_2}\frac{
-\frac{1}{\lambda_2}\frac{e^{-\lambda_2 (y_1-a)}}{e^{\lambda_1 (y_2-a)}}e^{\lambda_1 (D_1(A^*(B),B)-a)}
+\frac{1}{\lambda_2}e^{-\lambda_2 (D_1(A^*(B),B)-a)}}{g_{A^*(B),B}'(D_1(A^*(B),B))}\\
&& \quad =\frac{2}{\sigma^2}\frac{1}{\lambda_1+\lambda_2}\frac{1}{\lambda_2}
e^{-\lambda_2 (y_1-a)}\frac{e^{\lambda_1 (D_1(A^*(B),B)-y_2)}
    -e^{-\lambda_2 (D_1(A^*(B),B)-y_1)}}{g_{A^*(B),B}'(D_1(A^*(B),B))}\\
&&\quad <0,
\end{eqnarray*}
where the inequality is due to $D_1(A^*(B),B)<U(A^*(B),B)<y_1$, $D_1(A^*(B),B)<U(A^*(B),B)<y_2$
and $g_{A^*(B),B}'(D_1(A^*(B),B))>0$.
Therefore, we have proved (\ref{eq:D1limitfinite}).

Finally we show that
$\frac{\partial \Lambda_1(A^*(B),B)}{\partial B}>0$. It follows from
\eqref{eq:dA^*(B)/dB-discounted}  that
\begin{eqnarray*}
&&\frac{\partial \Lambda_1(A^*(B),B)}{\partial B}\\
&&\quad=\int_{d(A^*(B),B)}^{D(A^*(B),B)}\bigl[\frac{\partial g_{A^*(B),B}(x)}{\partial B}+k\bigr]dx
+\frac{\partial D(A^*(B),B)}{\partial B}\bigl[g_{A^*(B),B}(D(A^*(B),B))+k\bigr]\\
&&\quad\quad-\frac{\partial d(A^*(B),B)}{\partial B}\bigl[g_{A^*(B),B}(d(A^*(B),B))+k\bigr] \\
&&\quad =\frac{2}{\sigma^2}\frac{1}{\lambda_1+\lambda_2}\int_{d(A^*(B),B)}^{D(A^*(B),B)}
\bigl[\frac{1}{\lambda_1}\frac{dA^*(B)}{d B}e^{\lambda_1 (x-a)}+\frac{1}{\lambda_2}e^{-\lambda_2 (x-a)}\bigr]dx\\
&&\quad =\frac{2}{\sigma^2}\frac{1}{\lambda_1+\lambda_2}\int_{d(A^*(B),B)}^{D(A^*(B),B)}
\bigl[\frac{1}{\lambda_1}\frac{\lambda_1^2(e^{-\lambda_2 (u(A^*(B),B)-a)}-e^{-\lambda_2 (U(A^*(B),B)-a)})}
{\lambda_2^2(e^{\lambda_1 (u(A^*(B),B)-a)}-e^{\lambda_1 (U(A^*(B),B)-a)})}e^{\lambda_1 (x-a)}\\
&&\quad\quad+\frac{1}{\lambda_2}e^{-\lambda_2 (x-a)}\bigr]dx\\
&&\quad
=\frac{2}{\sigma^2}\frac{1}{\lambda_1+\lambda_2}\frac{1}{\lambda_2^2}
\frac{1}{e^{\lambda_1 (u(A^*(B),B)-a)}-e^{\lambda_1 (U(A^*(B),B)-a)}} \Big[\\
&&\quad\quad{(e^{-\lambda_2 (u(A^*(B),B)-a)}-e^{-\lambda_2 (U(A^*(B),B)-a)})
(e^{\lambda_1 (D(A^*(B),B)-a)}-e^{\lambda_1 (d(A^*(B),B)-a)})}
\\
&&\quad\quad-{(e^{-\lambda_2 (D(A^*(B),B)-a)}-e^{-\lambda_2 (d(A^*(B),B)-a)})(e^{\lambda_1 (u(A^*(B),B)-a)}-e^{\lambda_1 (U(A^*(B),B)-a)})}
\Big].
\end{eqnarray*}
If the expression inside the bracket is positive,
% \begin{eqnarray}
% &&(e^{-\lambda_2 (u(A^*(B),B)-a)}-e^{-\lambda_2 (U(A^*(B),B)-a)})(e^{\lambda_1 (D(A^*(B),B)-a)}-e^{\lambda_1 (d(A^*(B),B)-a)})\nonumber\\
% &&\quad-(e^{-\lambda_2 (D(A^*(B),B)-a)}-e^{-\lambda_2 (d(A^*(B),B)-a)})(e^{\lambda_1 (u(A^*(B),B)-a)}-e^{\lambda_1 (U(A^*(B),B)-a)})\nonumber\\
% &&\quad>0,\label{eq:dLambda1>0-discounted}
% \end{eqnarray}
we must have $\frac{\partial \Lambda_1(A^*(B),B)}{\partial B}>0$.
Note that $d(A^*(B),B)<D(A^*(B),B)<U(A^*(B),B)<u(A^*(B),B)$. Thus, the
positivity of the expression is equivalent to
\begin{eqnarray}
\label{eq:dLambda1>02-discounted}
\frac{e^{-\lambda_2 (u(A^*(B),B)-a)}-e^{-\lambda_2 (U(A^*(B),B)-a)}}{e^{\lambda_1 (u(A^*(B),B)-a)}-e^{\lambda_1 (U(A^*(B),B)-a)}}
>\frac{e^{-\lambda_2 (D(A^*(B),B)-a)}-e^{-\lambda_2 (d(A^*(B),B)-a)}}{e^{\lambda_1 (D(A^*(B),B)-a)}-e^{\lambda_1 (d(A^*(B),B)-a)}}.
\end{eqnarray}
Using the {\em Lagrange Mean Value Theorem}, there exist
$z_1\in(d(A^*(B),B),D(A^*(B),B))$ and $z_2\in(d(A^*(B),B),D(A^*(B),B))$ such that
\begin{eqnarray*}
&&e^{-\lambda_2 (D(A^*(B),B)-a)}-e^{-\lambda_2 (d(A^*(B),B)-a)}=-\lambda_2e^{-\lambda_2 (z_1-a)}(D(A^*(B),B)-d(A^*(B),B)),\\
&&e^{\lambda_1 (D(A^*(B),B)-a)}-e^{\lambda_1 (d(A^*(B),B)-a)}=\lambda_1e^{\lambda_1 (z_2-a)}(D(A^*(B),B)-d(A^*(B),B)).
\end{eqnarray*}
Using \eqref{eq:y1-discounted} and \eqref{eq:y2-discounted}, we have
that inequality
\eqref{eq:dLambda1>02-discounted} is equivalent to
\begin{eqnarray*}
\frac{e^{-\lambda_2 (y_1-a)}}{e^{\lambda_1 (y_2-a)}}
<\frac{e^{-\lambda_2 (z_1-a)}}{e^{\lambda_1 (z_2-a)}},
\end{eqnarray*}
which is further equivalent to
\begin{eqnarray}
\label{eq:dLambda1>03-discounted}
e^{-\lambda_2 (y_1-z_1)}<e^{\lambda_1 (y_2-z_2)}.
\end{eqnarray}
Inequality (\ref{eq:dLambda1>03-discounted}) holds  because
$y_1>U(A^*(B),B)>D(A^*(B),B)>z_1$ and
$y_2>U(A^*(B),B)>D(A^*(B),B)>z_2$ imply that
\begin{eqnarray*}
y_1-z_1>0,\quad y_2-z_2>0.
\end{eqnarray*}
Therefore, we have proved $\frac{\partial
  \Lambda_1(A^*(B),B)}{\partial B}>0$, completing the proof of the lemma.
\end{proof}

% \section{Conclusions}
% \label{sec:conclusions}
% In this paper, we use a common method to study the optimality of
% control band policy for three Brownian motion control problems with
% general convex inventory cost function and demonstrate how to find the
% optimal parameters. For the future research, it would be interesting
% to study the multi-stage system with Brownian motion demand.  Yao
% \cite{Yao10} has done a preliminary study on this filed.

\section*{Acknowledgments}
The authors would like to thank Hanqin Zhang at Chinese Academy of Sciences
and National University of Singapore for stimulating discussions.  Part of
work was done when the second author visited School of Industrial and
Systems Engineering, Georgia Institute of Technology, and the author
would like to thank the hospitality of the school.

\bibliography{dai}

\end{document}